\documentclass[11pt]{amsart}
\usepackage[utf8]{inputenc}

\pdfoutput=1

\usepackage{parskip}
\usepackage{mathrsfs}

\usepackage[margin=1in]{geometry}

\usepackage{enumerate}
\usepackage{amsthm}
\usepackage{amsmath}
\usepackage{amsfonts}
\usepackage{amssymb}
\usepackage{graphicx}
\usepackage{color}
\usepackage{graphics}
\usepackage{eepic}
\usepackage{nicefrac}

\makeatletter
\def\thm@space@setup{%
  \thm@preskip=\parskip \thm@postskip=0pt
}
\makeatother

\usepackage{tikzsymbols}

\newcommand{\ignore}[1]{}

\usepackage[colorinlistoftodos]{todonotes}

\newcommand{\be}{\begin{equation}}
\newcommand{\ee}{\end{equation}}

\renewcommand{\Re}{\operatorname{Re}}
\renewcommand{\Im}{\operatorname{Im}}

\newcommand{\C}{{\mathbb{C}}}
\newcommand{\R}{{\mathbb{R}}}

\newcommand{\D}{{\mathbb{D}}}

\newcommand{\E}{{\mathbb{E}}}
\newcommand{\EE}{{\mathbb{E}}}
\newcommand{\PP}{{\mathbb{P}}}

\usepackage{mathtools} 
\mathtoolsset{showonlyrefs,showmanualtags}

\newcommand{\diam}{\mbox{diam}}

\newtheorem{thm}{Theorem}[section]

\newtheorem{prop}[thm]{Proposition}

\newtheorem{cor}[thm]{Corollary}

\newtheorem{lemma}[thm]{Lemma}

\newtheorem*{claim*}{Claim}

\theoremstyle{definition}

\newtheorem{example}[thm]{Example}

\theoremstyle{remark}

 \renewcommand\epsilon{\varepsilon}

\usepackage{palatino}
 \usepackage{color, xcolor, mathrsfs, graphicx, hyperref, framed}

\title{Inradius of random lemniscates}
\author{Manjunath Krishnapur, Erik Lundberg, and Koushik Ramachandran}

\begin{document}

\begin{abstract}
A classically studied geometric property associated to a complex polynomial $p$ is the inradius (the radius of the largest inscribed disk) of its (filled) lemniscate $\Lambda := \{z \in \mathbb{C}:|p(z)| < 1\}$.

In this paper, we study the lemniscate inradius when the defining polynomial $p$ is random, namely, with the zeros of $p$ sampled independently from a compactly supported probability measure $\mu$.  If the negative set of the logarithmic potential $U_{\mu}$ generated by $\mu$ is non-empty, then the inradius is bounded from below by a positive constant with overwhelming probability.  Moreover, the inradius has a determinstic limit if the negative set of $U_{\mu}$ additionally contains the support of $\mu$.

On the other hand, when the zeros are sampled independently and uniformly from the unit circle, then the inradius converges in distribution to a random variable taking values in $(0,1/2)$.

We also consider the characteristic polynomial of a Ginibre random matrix whose lemniscate we show is close to the unit disk with overwhelming probability.

\end{abstract}

\maketitle

\section{Introduction}
\noindent Let $p(z)$ be a polynomial of degree $n$ and $\Lambda$ be its (filled) lemniscate defined by $\Lambda= \{z: |p(z)| < 1\}.$
Denote by $\rho(\Lambda)$ the inradius of $\Lambda$. By definition, this is the radius of the largest disk that is completely contained in $\Lambda.$ 
In this paper, we study the inradius of random lemniscates for various models of random polynomials.

The lemniscate $\{z:|z^n-1| < 1 \}$ has an inradius asymptotically proportional to $1/n$.
In 1958, P. Erd\"os, F. Herzog, and G. Piranian posed a number of problems \cite{EHP} on geometric properties of polynomial lemniscates.  Concerning the inradius, they asked \cite[Problem 3]{EHP} whether the rate of decay in the example $\{|z^n-1|=1 \}$ is extremal, that is, whether there exists a positive constant $C$ such that for any monic polynomial of degree $n,$ all of whose roots lie in the closed unit disk, the inradius $\rho$ of its lemniscate $\Lambda$ satisfies $\rho\geq\frac{C}{n}$.
This question remains open.
C. Pommerenke \cite{Pomm1961} showed in this context that
the inradius satisfies the lower bound $\rho \geq \frac{1}{2e\, n^2}$. 

%In another direction, A. Solynin and A. Williams \cite{Solynin} showed, confirming a conjecture of H. Cuenya and F. Levis \cite{Cuenya}, that for each $n$ there exists a constant $C(n)$ such that the inradius $\rho(\Lambda)$ satisfies the estimate $\rho(\Lambda) \geq C(n) \sqrt{A(\Lambda)} $, where $A(\Lambda)$ denotes the area of $\Lambda$.

 Our results, which we state below in Sec. \ref{sec:results} of the Introduction, show within probabilistic settings that the \emph{typical} lemniscate admits a much better lower bound on its inradius.
Namely, if the zeros of $p$ are sampled independently from a compactly supported measure $\mu$ whose logarithmic potential has non-empty negative set, then the inradius of $\Lambda$ is bounded below by a positive constant with overwhelming probability, see Theorem \ref{G} below.
Let us provide some insight on this result and explain why the logarithmic potential of $\mu$ plays an important role.  First, the lemniscate $\Lambda$ can alternatively be described as the sublevel set $\{ \frac{1}{n}\log |p(z)| < 0 \}$ of the discrete logarithmic potential $\frac{1}{n} \log|p(z)| = \frac{1}{n}\sum \log|z-z_k|$ where $z_k$ are the zeros of $p(z)$.  For fixed $z$ the sum $\frac{1}{n}\sum \log|z-z_k|$ is a Monte-Carlo approximation for the integral defining the logarithmic potential $U_\mu(z)$ of $\mu$, and, in particular, it converges pointwise, by the law of large numbers, to $U_\mu(z)$.  With the use of large deviation estimates, we can further conclude that each $z$ in the negative set $\Omega^-$ of $U_\mu$ is in $\Lambda$ with overwhelming probability.  The property of holding with overwhelming probability survives (by way of a union bound) when taking an intersection of polynomially many such events.  This fact, together with a suitable uniform estimate for the derivative $p'(z)$ (for which we can use a Bernstein-type inequality), allows for a standard epsilon-net argument showing that an arbitrary compact subset of $\Omega^-$ is contained in $\Lambda$ with overwhelming probability.  Since $\Omega^-$ is assumed nonempty, this leads to the desired lower bound on the inradius, see the proof of Theorem \ref{G} in Section \ref{sec:main} for details.

Under an additional assumption that the negative set $\Omega^-$ of the logarithmic potential of $\mu$ contains the support of $\mu$, the inradius converges to the inradius of $\Omega^-$ almost surely, see Corollary \ref{cor:convergenceofinradius}; in particular, the inradius has a deterministic limit.

On the other hand, for certain measures $\mu$, the inradius does not have a deterministic limit and rather converges in distribution to a nondegenerate random variable, see Theorem \ref{thm:unifcirc} addressing the case when $\mu$ is uniform measure on the unit circle.
We also consider the lemniscate associated to the characteristic polynomial of a random matrix sampled from the Ginibre ensemble, and we show that the inradius is close to unity (in fact the whole lemniscate is close to the unit disk) with overwhelming probability, see Theorem \ref{Gin}.

See Section \ref{sec:results} below for precise statements of these results along with some additional results giving further insight on the geometry of $\Lambda$.

\vspace{0.1in}

%The above-mentioned paper \cite{EHP} of Erd\"os, Herzog, and Piranian stated a number of problems on geometric properties of polynomial lemniscates.

\subsection{Previous results on random lemniscates}

The current paper fits into a series of recent studies investigating the geometry and topology of random lemniscates.  Let us summarize previous results in this direction.  We note that the lemniscates studied in the results cited below, in contrast to the filled lemniscates of the current paper, are level sets (as opposed to sublevel sets).

Partly motivated to provide a probabilistic counterpart to the Erd\"os lemniscate problem on the extremal length of lemniscates \cite{EHP}, \cite{Borwein}, \cite{EremHay}, \cite{FN}, the second and third authors in \cite{LR} studied the arclength and topology of a random polynomial lemniscate in the plane.  When the polynomial has i.i.d. Gaussian coefficients, it is shown in \cite{LR} that the average length of its lemniscate approaches a constant.
They also showed that with high probability the length is bounded by a function with arbitrarily slow rate of growth, which means that the length of a lemniscate typically satisfies a much better estimate than the extremal case.  
It is also shown in \cite{LR} that the number of connected components of the lemniscate is asymptotically $n$ (the degree of the defining polynomial) with high probability, and there is at least some fixed positive probability of the existence of a ``giant component'', that is, a component having at least some fixed positive length.
Of relevance to the focus of the current paper, we note that the proof of the existence of the giant component in \cite{LR} shows that for a fixed $0<r<1$, there is a positive probability that the inradius $\rho$ of the lemniscate satisfies the lower bound $\rho > r$.

Inspired by Catanese and Paluszny's topological classification \cite{Catanese} of generic polynomials (in terms of the graph of the modulus of the polynomial with equivalence up to diffeomorphism of the domain and range), in \cite{EHL} the second author with  M. Epstein and B. Hanin studied the so-called lemniscate tree associated to a random polynomial of degree $n$.  The lemniscate tree of a polynomial $p$ is a labelled, increasing, binary, nonplane tree that encodes the nesting structure of the singular components of the level sets of the modulus $|p(z)|$.  When the zeros of $p$ are i.i.d. sampled uniformly at random according to a probability density that is bounded with respect to Haar measure on the Riemann sphere, it is shown in \cite{EHL} that the number of branches (nodes with two children) in the induced lemniscate tree is $o(n)$ with high probability, whereas a lemniscate tree sampled uniformly at random from the combinatorial class has asymptotically $\left( 1- \frac{2}{\pi} \right) n$  many branches on average.

In \cite{LLlemni}, partly motivated by a known result \cite{EremHay},  \cite{TrefWeg}) stating that the maximal length of a rational lemniscate on the Riemann sphere is $2\pi n$, the second author with A. Lerario studied the geometry of a random rational lemniscate and showed that the average length on the Riemann sphere is asymptotically $\frac{\pi^2}{2}\sqrt{n}$.    Topological properties (the number of components and their nesting structure) were also considered in \cite{LLlemni}, where the number of connected components was shown to be asymptotically bounded above and below by positive constants times $n$.   Z. Kabluchko and I. Wigman subsequently established an asymptotic limit law for the number of connected components in \cite{KaWi} by adapting a method of F. Nazarov and M. Sodin \cite{NazSod} using an integral geometry sandwich and ergodic theory applied to a  translation-invariant ensemble of planar meromorphic lemniscates obtained as a scaling limit of the rational lemniscate ensemble.

\subsection{Motivation for the study of lemniscates}
The study of lemniscates has a long and rich history with a wide variety of applications.  The problem of computing the length of Bernoulli's lemniscate played a role in the early study of elliptic integrals \cite{Ayoub}. Hilbert's lemniscate theorem and its generalizations \cite{Totik} show that lemniscates can be used to approximate rather arbitrary domains, and this density property contributes to the importance of lemniscates in many of the applications mentioned below. 
In some settings, sequences of approximating lemniscates arise naturally for example in holomorphic dynamics \cite[p. 159]{Milnor}, where it is simple to construct a nested sequence of ``Mandelbrot lemniscates'' that converges to the Madelbrot set.
In the classical inverse problem of logarithmic potential theory---to recover the shape of a two-dimensional object with uniform mass density from the logarithmic potential it generates outside itself---uniqueness has been shown to hold for lemniscate domains \cite{Strakhov}.  This is perhaps surprising in light of Hilbert's lemniscate theorem and the fact that the inverse potential problem generally suffers from non-uniqueness \cite{Etingof}.  Since lemniscates are real algebraic curves with useful connections to complex analysis, they have frequently received special attention in studies of real algebraic curves, for instance in the study of the topology of real algebraic curves \cite{Catanese}, \cite{BauerCatanese}.   Leminscates such as the Arnoldi lemniscate appear in applications in numerical analysis \cite{Tref}.  Lemniscates have seen applications in two-dimensional shape compression, where the ``fingerprint'' of a shape constructed from conformal welding simplifies to a particularly convenient form---namely the $n$th root of a  Blaschke product---in the case the two-dimensional shape is assumed to be a degree-$n$ lemniscate \cite{EKS}, \cite{Younsi}, \cite{RichardsYounsi}.
Lemniscates have appeared in studies of moving boundary problems of fluid dynamics \cite{KhavLund}, \cite{LundTotik}, \cite{KMPT}.
In the study of planar harmonic mappings, rational lemniscates arise as the critical sets of harmonic polynomials \cite{KhavLeeSaez}, \cite{LLtrunc} as well as critical sets of lensing maps arising in the theory of gravitational lensing \cite[Sec. 15.2.2]{Petters}.
Lemniscates also have appeared prominently in the theory and application of conformal mapping \cite{Bell}, \cite{JT}, \cite{GPSS}.
See also the recent survey \cite{Richards} which elaborates on some of the more recent of the above mentioned lines of research.

\subsection{Definitions and Notation}
Throughout the paper,  $\mu$ will denote a Borel probability measure with compact support $S\subset\C$. The logarithmic potential of $\mu$ is defined by
$$U_{\mu}(z)= \int_{S}\log|z-w|d\mu(w). $$ 
It is well known that $U_{\mu}$ is a subharmonic function in the plane, and harmonic in $\C\setminus S$. For such $\mu$, we denote the associated negative and positive sets of its potential by
$$\Omega^{-} = \{z\in\C: U_{\mu}(z) < 0\}, \hspace{0.05in} \Omega^{+}= \{z\in\C: U_{\mu}(z) > 0\}.$$ 
It is easy to see that $\Omega^{-}$ is a (possibly empty) bounded open set.

\noindent{\bf Assumptions on the measure.}
\noindent Let $\mu$ be a Borel probability measure with compact support $S\subset\C$. We define the following progressively stronger conditions on $\mu$.
\begin{enumerate}[(A)]
\item  \label{A} For each compact $K\subset\C$,
\begin{equation*}
%\label{ubvar}
C(K) =\sup_{z\in K}\int_{S}\left(\log |z-w|\right)^2d\mu(w) < \infty .  
\end{equation*}
\item  \label{B} There is some $C<\infty$ and $\epsilon>0$ such that for all $z\in\C$ and all $r\leq 1,$ we have $$\mu\left(B(z, r)\right)\leq \frac{C}{(\log (1/r))^{2 +\epsilon}}.$$
\item  \label{C} There exists $\delta > 0$ such that 
$$\sup_{z\in\mathbb{C}}\int_{S}\dfrac{d\mu(w)}{|z-w|^{\delta}} < \infty .$$
\item  \label{D} There is some $C<\infty$ and $\epsilon>0$ such that for all $z\in\C$ and all $r>0$, we have
$$\mu\left(B(z, r)\right)\leq Cr^{\epsilon}.$$
\end{enumerate}

%%\begin{defn}\label{A}
%%\noindent Let $\mu$ be a Borel probability measure with compact support $S\subset\C.$ We say that $\mu$ satisfies assumption $(A)$ if for each compact $K\subset\C$,
%%\begin{equation}\label{ubvar}
%%C(K) =\sup_{z\in K}\int_{S}\left(\log |z-w|\right)^2d\mu(w) < \infty .  
%%\end{equation}
%%\end{defn}
%%
%%\vspace{0.1in}
%%
%%\noindent A sufficient condition for Assumption $(A)$ to hold is that for all $z\in\C,$ and all $r\leq 1,$
%%$$\mu\left(B(z, r)\right)\leq \frac{C_z}{(\log (1/r))^{2 +\epsilon}},$$
%%for some $\epsilon >0,$ with the constants $C_z$ being locally bounded. 
%%
%%\vspace{0.1in}
%%
%%\begin{defn}\label{B}
%%\noindent We say that $\mu$ satisfies assumption $(B)$ if there exists $\delta > 0$ such that 
%%$$\sup_{z\in\mathbb{C}}\int_{S}\dfrac{d\mu(w)}{|z-w|^{\delta}} < \infty .$$
%%\end{defn}

\vspace{0.1in}

%\begin{defn}
%Suppose that $p(z)$ is a polynomial in one complex variable. Then the (unit) lemniscate of $p$ is defined 
%$$\Lambda = \Lambda(p) = \{z\in\C: |p(z)| < 1\}$$
%\end{defn}
%
%We denote by $\rho = \rho(\Lambda)$ the inradius of $\Lambda.$ This is the radius of the largest open disk that is contained in $\Lambda.$ For a sequence of polynomials $\{p_n\},$ we denote by $\Lambda_n$ and $\rho_n$ the corresponding lemniscates and their inradii respectively.

\subsection{Main results}\label{sec:results}
In all theorems (except Theorem~\ref{Gin}), we have the following setting:

\noindent{\bf Setting}: $\mu$ is a compactly supported probability measure on $\C$ with support $S$. The random variables $X_i$ are i.i.d. from the distribution $\mu$. We consider the random polynomial $p_n(z):=(z-X_1)\ldots (z-X_n)$ and its lemniscate $\Lambda_n:=\{z \; : \; |p_n(z)|<1\}$. We write $\rho_n=\rho(\Lambda_n)$ for the inradius of $\Lambda_n$.

Throughout the paper, w.o.p. means with overwhelming probability, i.e., with probability at least $1-e^{-cn}$ for some $c>0$.

The theorems below concern the random lemniscate $\Lambda_n$. Observe that $\Lambda_n$ consists of all $z$ for which $\log|p_n(z)|<0$, or what is the same,
\begin{align*}
\frac{1}{n}\sum_{k=1}^n\log|z-X_k|<0.
\end{align*}
By the law of large numbers, the quantity on the left converges to $U_{\mu}(z)$ pointwise.  Hence we may expect the asymptotic behaviour of $\Lambda_n$ to be described in terms of $U_{\mu}$ and its positive and negative sets $\Omega^{+},\Omega^{-}$. The first three theorems make this precise under different conditions on the underlying measure $\mu$.

\begin{thm}\label{G}
 Assume that  $\mu$ satisfies assumption~(\ref{A}). Suppose that $\Omega^{-} \neq\emptyset$ and let $\rho = \rho(\Omega^{-}).$ 
 %Let $\{X_i\}$ be a sequence of random variables that are i.i.d. with distribution $\mu$. Define random polynomials $\{p_n\}$ by
 %$$p_n(z) = \prod_{k=1}^{n} (z-X_k).$$
%Denote by $\Lambda_n$ the corresponding lemniscate. 
Fix compact sets $K\subset\Omega^{-},$ and $L\subset\Omega^{+}\setminus{S}.$ Then for all large $n$,
$$K\subset\Lambda_n,\quad\mbox{w.o.p.,} \quad\mbox{and} \quad L\subset \Lambda_n^{c} \quad\mbox{w.o.p}. $$ 
In particular, if $\rho_n$ denotes the inradius  of $\Lambda_n$, then
$$\rho_n\geq a \quad\mbox{w.o.p.,}\quad\forall{ a\in (0, \rho) }$$
\end{thm}
\begin{cor}\label{cor:convergenceofinradius} In the setting of Theorem~\ref{G}, $\liminf \rho_{n}\ge \rho$ a.s. Further, if $S\subseteq \Omega^{-}$, then $\rho_n\to \rho$ a.s.
\end{cor}

%%{\color{red} Does the statement that $\rho_n\to \rho$ a.s. need the assumption that $S\subseteq \Omega^{-}$? Either find an example where $\liminf \rho_n>\rho$ or find a proof that works more generally.
%%
%%
%%\bigskip
%%
%%We expected a stronger result thinking that  $\Omega^{-}\not=\emptyset$ implies that $\Omega_0$ cannot have an interior. This is not clear, nor is it entirely clear if the proof goes through if we assume this (because of the need for lower bound on $U_{\mu}$ in the second part of Theorem~\ref{G}. But if true, that would show that $\rho_n\to \rho$ always.
%%}

Ideally, we would have liked to say that a compact set $L\subseteq \Omega^+$ is contained inside $\Lambda_n^c$ w.o.p. However, this is clearly not true if some of the roots fall inside  $L$. Making the  stronger assumption $(D)$ on the measure and further assuming that $U_{\mu}$ is bounded below by a positive number on $L$,  we show that $L$ is almost entirely contained in $\Lambda_n^c$. 

\begin{thm}\label{thm:pospot}
 Let $\mu$ satisfy assumption~(\ref{D}). Let $L$ be a compact subset of $\{U_{\mu}\ge m\}$ for some $m>0$. 
% Suppose that 
% $\inf_{z \in S} U_\mu(z) \geq m > 0$.
% Let $\{X_i\}$ be a sequence of random variables that are i.i.d. with distribution $\mu$. Define random polynomials  $\{p_n\}$  by
% $$p_n(z) = \prod_{k=1}^{n} (z-X_k).$$
%Denote by $\Lambda_n$ the corresponding lemniscate.  
Then there exists $c_0>0$ such that $$\Lambda_n\cap L \subset \bigcup_{k=1}^n B(X_k,e^{-c_0n}), \quad \text{w.o.p.}$$
\end{thm}

In particular, if $U_{\mu}\ge m$ everywhere, then the whole lemniscate is small. It suffices to assume that $U_{\mu}\ge m$ on the support of $\mu$, by the minimum principle for potentials (Theorem~3.1.4 in \cite{ransford}). 
\begin{cor}\label{cor:smalllemniscate} 
Suppose $\mu$ satisfies assumption~(\ref{D}) and $U_{\mu}\ge m$ on $S$. Then there is a $c_0>0$ such that $\Lambda_n\subset \bigcup_{k=1}^n B(X_k,e^{-c_0n})$ and  $\rho_n\le ne^{-c_0n}$ w.o.p.
% the inradius $\rho_n$ of $\Lambda_n$ is bounded by an exponentially decaying function of $n$ with overwhelming probability.
\end{cor}

%%What happens when $\Omega^{-}$ is empty? We do not have a complete answer, but under the stronger assumption that $U_{\mu}$ is bounded below by a positive constant,  the following theorem shows that the lemniscate is very small.
%%\begin{thm}\label{thm:pospot}
%% Let $\mu$ satisfy assumption~(\ref{D}). Suppose that 
%% $\inf_{z \in S} U_\mu(z) \geq m > 0$.
%%% Let $\{X_i\}$ be a sequence of random variables that are i.i.d. with distribution $\mu$. Define random polynomials  $\{p_n\}$  by
%%% $$p_n(z) = \prod_{k=1}^{n} (z-X_k).$$
%%%Denote by $\Lambda_n$ the corresponding lemniscate.  
%%Then there exists $c_0>0$ such that $$\Lambda_n \subset \bigcup_{k=1}^n B(X_k,e^{-c_0n}), \quad \text{w.o.p.}$$
%%In particular, $\rho_n\le ne^{-c_0n}$ w.o.p.
%%% the inradius $\rho_n$ of $\Lambda_n$ is bounded by an exponentially decaying function of $n$ with overwhelming probability.
%%\end{thm}

A class of examples illustrating Theorem~\ref{G} and Theorem~\ref{thm:pospot} is given at the end of the section.

What happens when the potential $U_{\mu}$ vanishes on a non-empty open set? In this case $\log|p_n|$ has zero mean, and is (approximately) equally likely to be positive or negative. Because of this, one may expect that the randomness in $\Lambda_n$ and $\rho_n$ persists in the limit and we can at best hope for a convergence in distribution. The particular case  when $\mu$ is uniform on the unit circle  is dealt with in the following theorem. 
%Unlike in the previous situations, the randomness in $\rho_n$ does not go away, and we can only get convergence in distribution.
\begin{thm}\label{thm:unifcirc}
Let $\mu$ be the uniform probability measure on $\mathbb{S}^1$,  the unit circle in the complex plane.  Then, $\rho_n\stackrel{d}{\rightarrow} \rho$ for some random variable $\rho$ taking values in $(0,\frac12)$. Further, $\mathbb P\{\rho<\epsilon\}>0$ and $\mathbb P\{\rho>\frac12 -\epsilon\}>0$ for every $\epsilon>0$.
\end{thm}

As shown in the proof of Theorem \ref{thm:unifcirc}, the random function $\log |p_n(z)|$ converges, after appropriate normalization, almost surely to a nondegenerate Gaussian random function on $\D$, and this convergence underlies the limiting random inradius $\rho$.
We note that similar methods can be used to study other measures $\mu$ for which $U_\mu$ vanishes on  non-empty open set (such as other instances where $\mu$ is the equilibrium measure of a region with unit capacity), however the case of the uniform measure on the circle is rather special, as the resulting random function $\log |p_n(z)|$ as well as its limiting Gaussian random function has a deterministic zero at the origin (which is responsible for the limiting inradius taking values only up to half the radius of $\D$.

Another setting where one can rely on convergence of the defining function $\log|p_n(z)|$ is in the case when the polynomial $p_n$ has i.i.d. Gaussian coefficients.  Actually, the convergence in this case is more transparent (and does not require additional tools such as Skorokhod's Theorem) as $p_n$ can already be viewed as the truncation of a power series with i.i.d. coefficients.  This case has a similar outcome as in Theorem \ref{thm:unifcirc}, except the value $1/2$ is replaced by $1$ due to the absence of a deterministic zero.

One can ask for results analogous to Theorems \ref{G} and \ref{thm:pospot} when the zeros are dependent random variables. A natural class of examples are determinantal point processes. We consider one special case here.

The Ginibre ensemble is a random set of $n$ points in $\C$ with joint density proportional to
\begin{align}\label{eq:ginijtdensity}
e^{-\sum_{k=1}^n|\lambda_k|^2} \prod_{j<k}|\lambda_j-\lambda_k|^2.
\end{align}
This arises in random matrix theory, as the distribution of eigenvalues of an $n\times n$ random matrix whose entries are i.i.d. standard complex Gaussian. After scaling by $\sqrt{n}$, the empirical distributions
\begin{equation}\label{esd}
\mu_n = \dfrac{1}{n}\sum_{j=1}^{n}\delta_{\frac{\lambda_j}{\sqrt{n}}}   
\end{equation} 
converge to the uniform measure on $\D$. Hence, we may expect the lemniscate of the corresponding polynomial to be similar to the case when the roots are sampled independently and uniformly from $\D$.
\begin{thm}\label{Gin}
Let $\lambda_1,\ldots ,\lambda_n$ have joint density given by \eqref{eq:ginijtdensity} and let $X_j = \frac{\lambda_j}{\sqrt{n}}$. Let $\Lambda_n$ be the  unit lemniscate of the random polynomial $p_n(z) = \prod_{k=1}^{n}(z-X_k)$. Given $r\in (0, 1)$ and $s\in (1, \infty)$, we have for large $n$,
$$\mathbb{D}_r\subseteq\Lambda_n\subseteq \D_{s}, \hspace{0.1in}\mbox{w.o.p.} $$
\end{thm}

\medskip
\begin{example}\label{example}
Let $\mu_r$ be the normalized area measure on the  disk $r\D$  and suppose the roots are sampled from $\mu_r$.  It is easy to check that $\mu_r$ satisfies assumptions $(A)-(D)$. We claim that
\be\label{eq:diskpot}
U_{\mu_r}(z) = \begin{cases} \frac{|z|^2-r^2}{2r^2}+\log r &\mbox{ if } |z|<r, \\ \log|z| & \mbox{ if }|z|\ge r. \end{cases}
%U_{\mu}(z) = \begin{cases} \frac{|z|^2-1}{2} &\mbox{ if } |z|<1, \\ \log|z| & \mbox{ if }|z|\ge 1. \end{cases}
\ee
Therefore, $\Omega^{-}=r_c\D$ where 
$$
r_c=\begin{cases} 1 & \mbox{ if }r\le 1, \\ r\sqrt{1-2\log r} & \mbox{ if }1\le r\le \sqrt{e}, \\ 0 &\mbox{ if }r\ge \sqrt{e}.\end{cases}
$$  
Hence Theorem \ref{G} implies that when $r<\sqrt{e}$, any disk $\D_s$ with radius $s <r_c$ is contained in $\Lambda_n$ with overwhelming probability as $n \rightarrow \infty$. When $r\le 1$, Corollary~\ref{cor:convergenceofinradius} implies that $\Lambda_n$ is almost the same as $\Omega^{-}=\D$. For $r>\sqrt{e}$, Theorem~\ref{thm:pospot} applies to show that $\Lambda_n$ is contained in a union of very small disks.

Let us carry out the computations to verify \eqref{eq:diskpot}. By rescaling, it is clear that $U_{\mu_r}(z)=\log r+U_{\mu_1}(z/r)$, hence it suffices to consider $r=1$. 
$$U_{\mu_1}(z)= \frac{1}{\pi}\int_{\D}\log|z-w|dA(w). $$
For $|z|\ge 1$, the integrand is harmonic with respect to $w\in \D$, hence $U_{\mu}(z)=\log|z|$ by the mean-value theorem. For $|z|<1$, we separate the integral over the two regions where $|w|<|z|$ and $|w|>|z|$.
Harmonicity of $w\mapsto \log|z-w|$ on $\{|w|<|z|\}$ and the mean-value property gives
%is harmonic with respect to $w$ for $|w|<|z|$, we have, by the mean value property for harmonic functions 
$$\int_{|w|<|z|}\log|z-w|dA(w) = \pi|z|^2 \log|z|. $$
We switch to polar coordinates $w=r e^{i\theta}$ for the second integral.
\begin{align*}
    \int_{|z|}^1\int_{0}^{2\pi}\log|z-re^{i\theta}| r d\theta dr &=   \int_{|z|}^1 \underbrace{\int_{0}^{2\pi}\log|z e^{-i\theta}-r| d\theta}_{2\pi \log r} r dr \\
    &=   \int_{|z|}^1 2\pi r \log r dr \\
    &=   2\pi \left( \frac{1}{4} - \frac{|z|^2}{2} \log |z| +\frac{|z|^2}{4} \right),
\end{align*}
where we have again used the mean value property (this time over a circle) for harmonic functions to compute the inside integral in the first line above.
Combining these integrals over the two regions and dividing by $\pi$ we arrive at \eqref{eq:diskpot}. 
\end{example}
\vspace{0.1in}

\begin{figure}[h]
\centering
\includegraphics[scale=0.4]{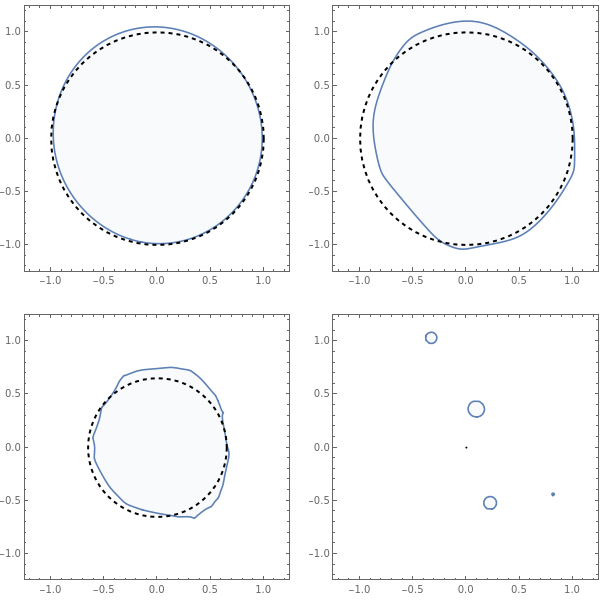}
\caption{Lemniscates of degree $n=30,40,400,15$ with zeros sampled uniformly from the  disks of radii $0.5,1,1.5,1.7$ (order: from  top-left to bottom-right). The dotted circle has radius $r_c$.}
\label{fig:disk1}
\end{figure}

%%\begin{figure}[h]
%%\centering
%%\includegraphics[scale=0.17]{disk20.pdf}
%%\includegraphics[scale=0.17]{disk30.pdf}
%%\includegraphics[scale=0.17]{disk40.pdf}
%%\caption{Lemniscates of degree $n=20,30,40$ with zeros sampled uniformly from the unit disk. }
%%\label{fig:disk}
%%\end{figure}

\begin{figure}[h]
\centering
\includegraphics[scale=0.17]{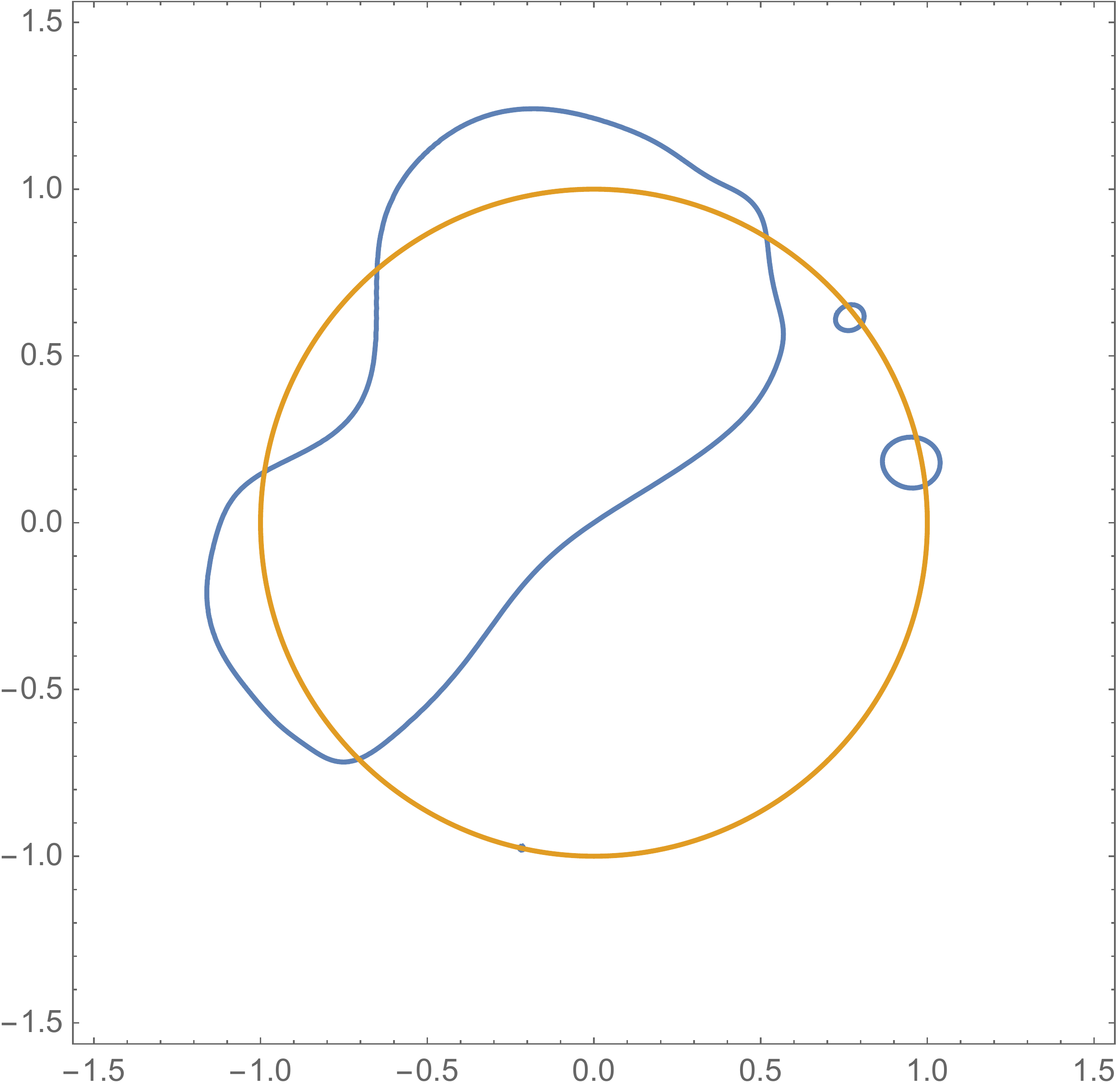}
\includegraphics[scale=0.17]{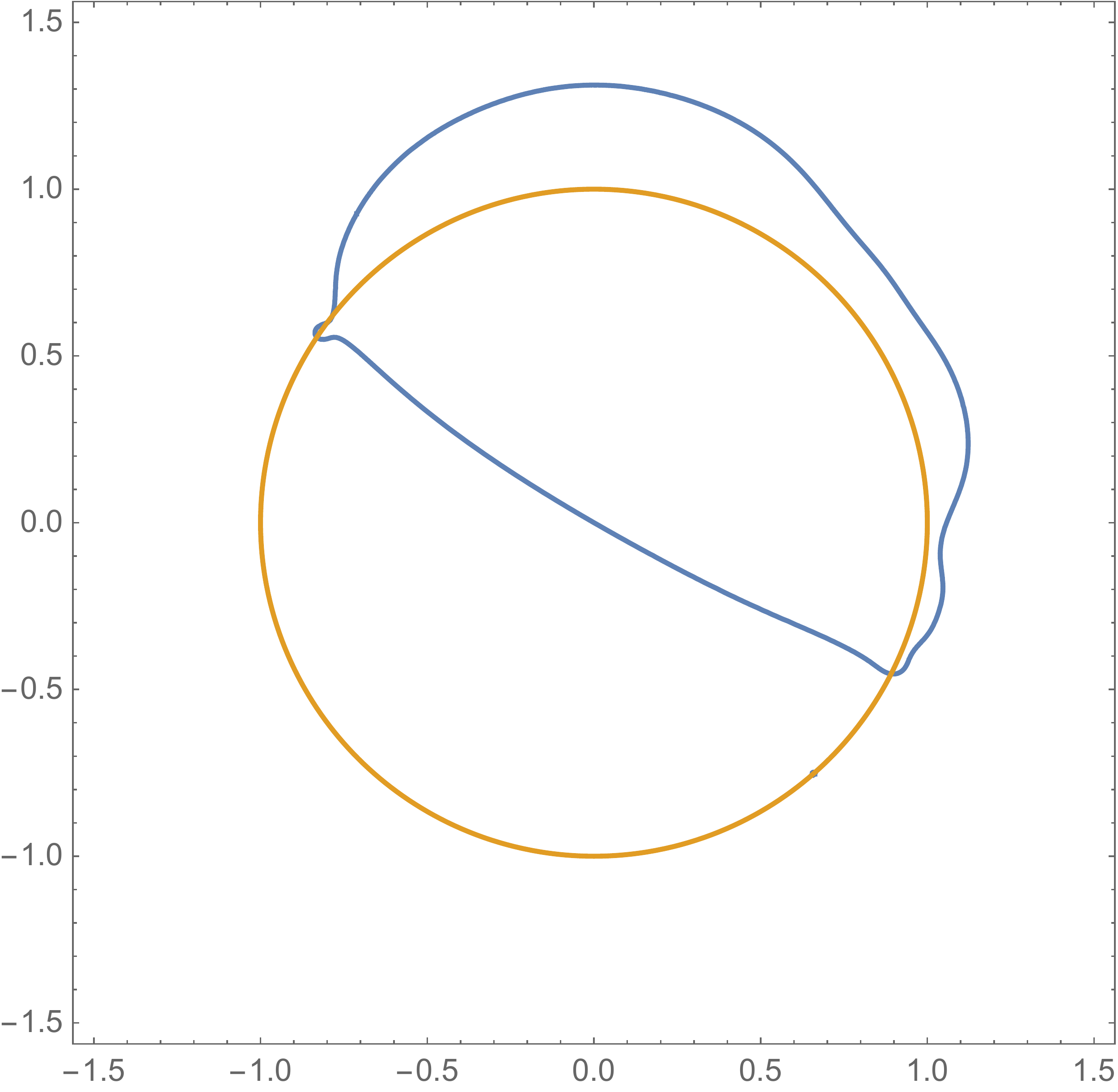}
\includegraphics[scale=0.17]{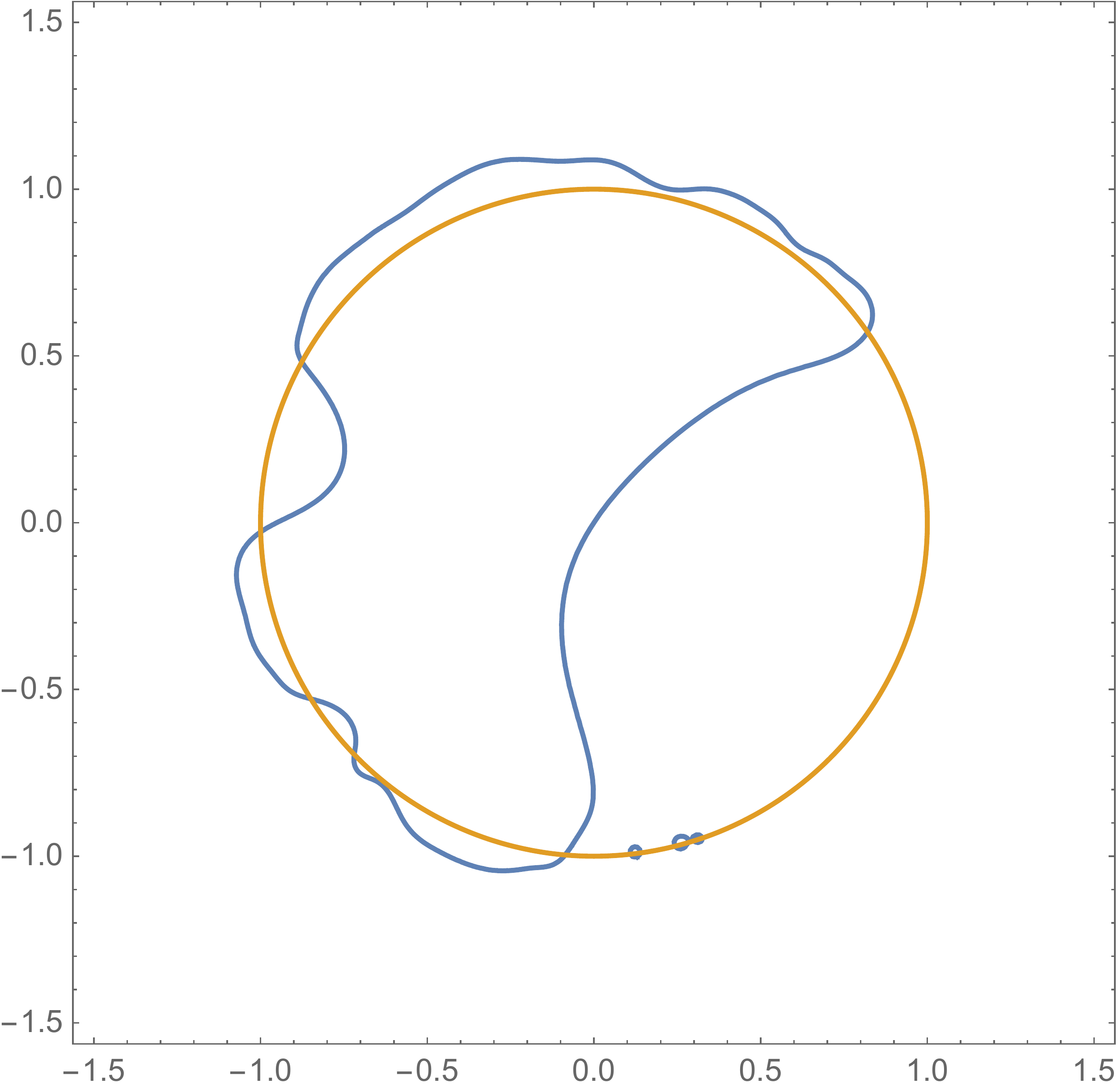}
\caption{Lemniscates of degree $n=20,30,40$ with zeros sampled uniformly from the unit circle. A unit circle is also plotted for reference in each case.}
\label{fig:circle}
\end{figure}

\begin{figure}[h]
\centering
\includegraphics[scale=0.17]{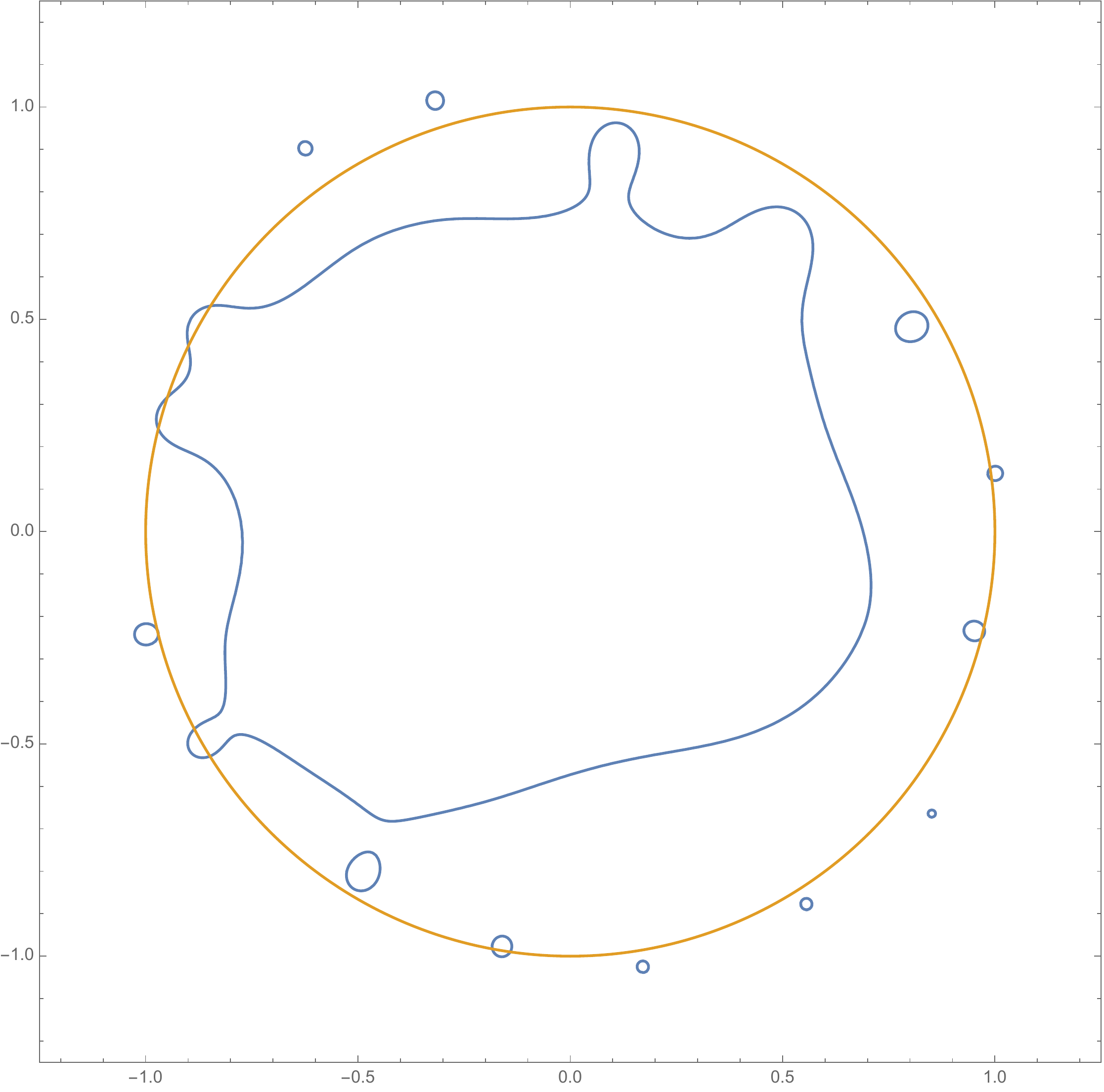}
\includegraphics[scale=0.17]{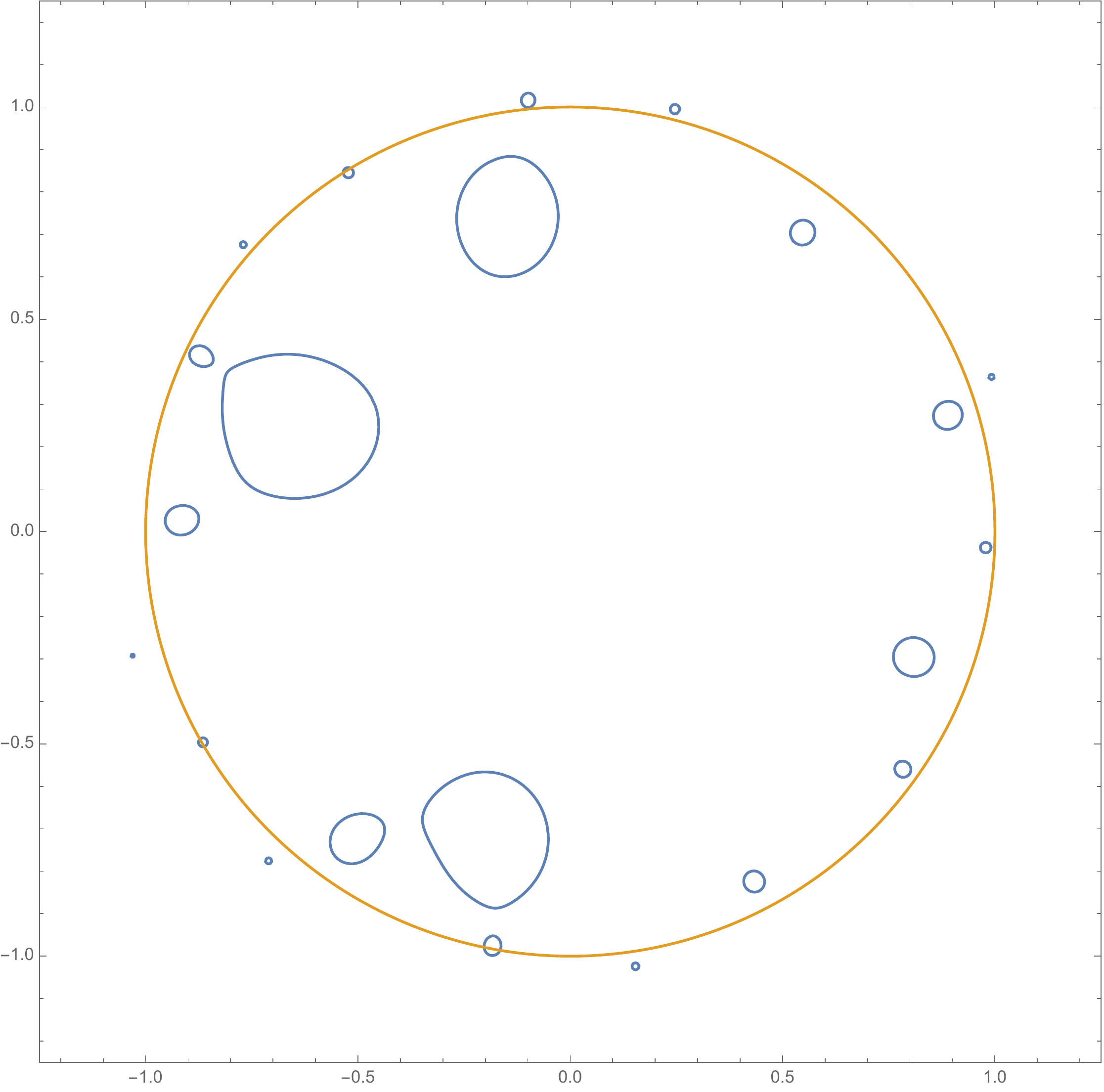}
\includegraphics[scale=0.17]{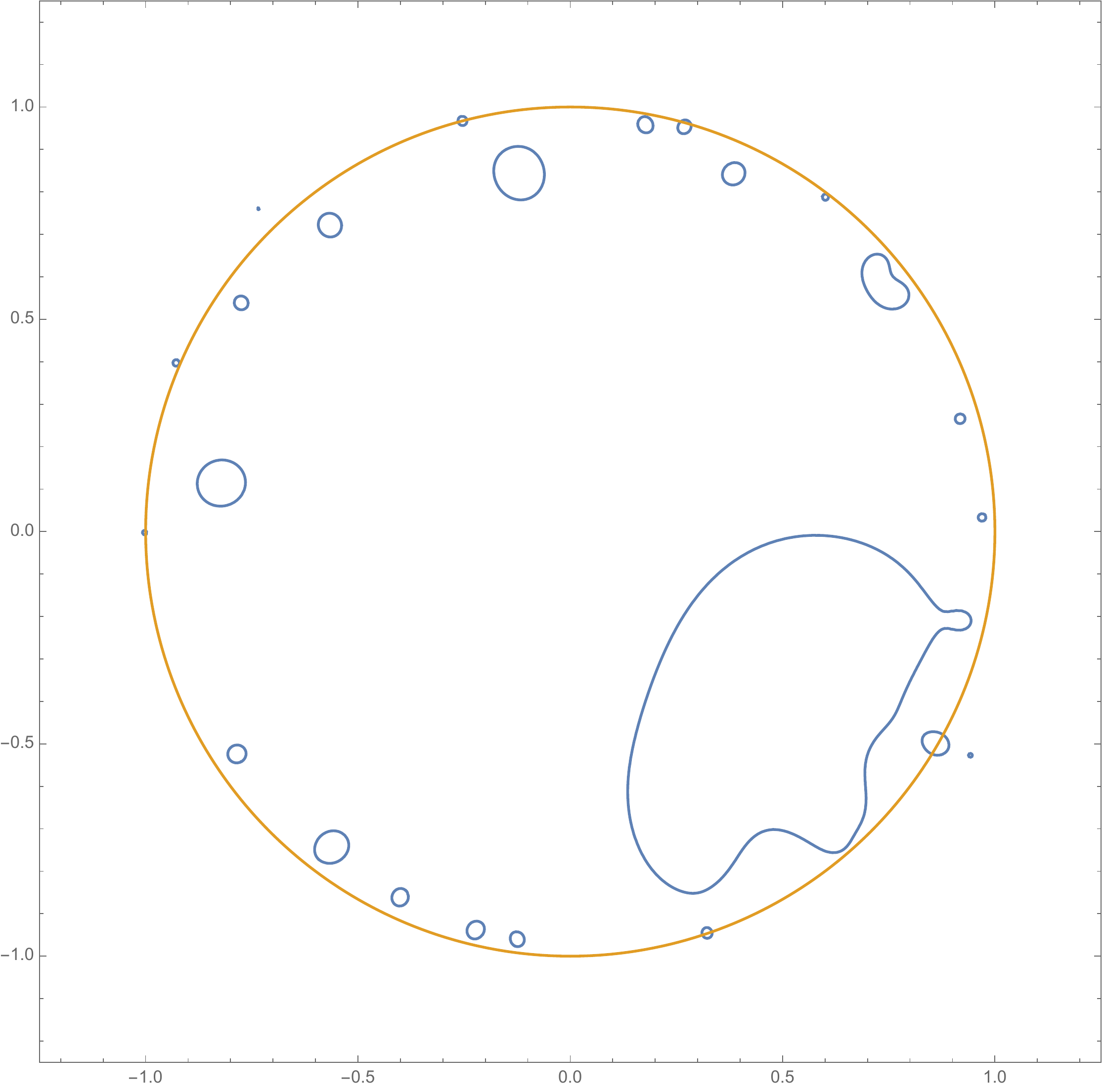}
\caption{Lemniscates of degree $n=20,30,40$  with i.i.d. Gaussian coefficients plotted together with a unit circle for reference.}
\label{fig:Kac}
\end{figure}

\begin{figure}[h]
\centering
\includegraphics[scale=0.17]{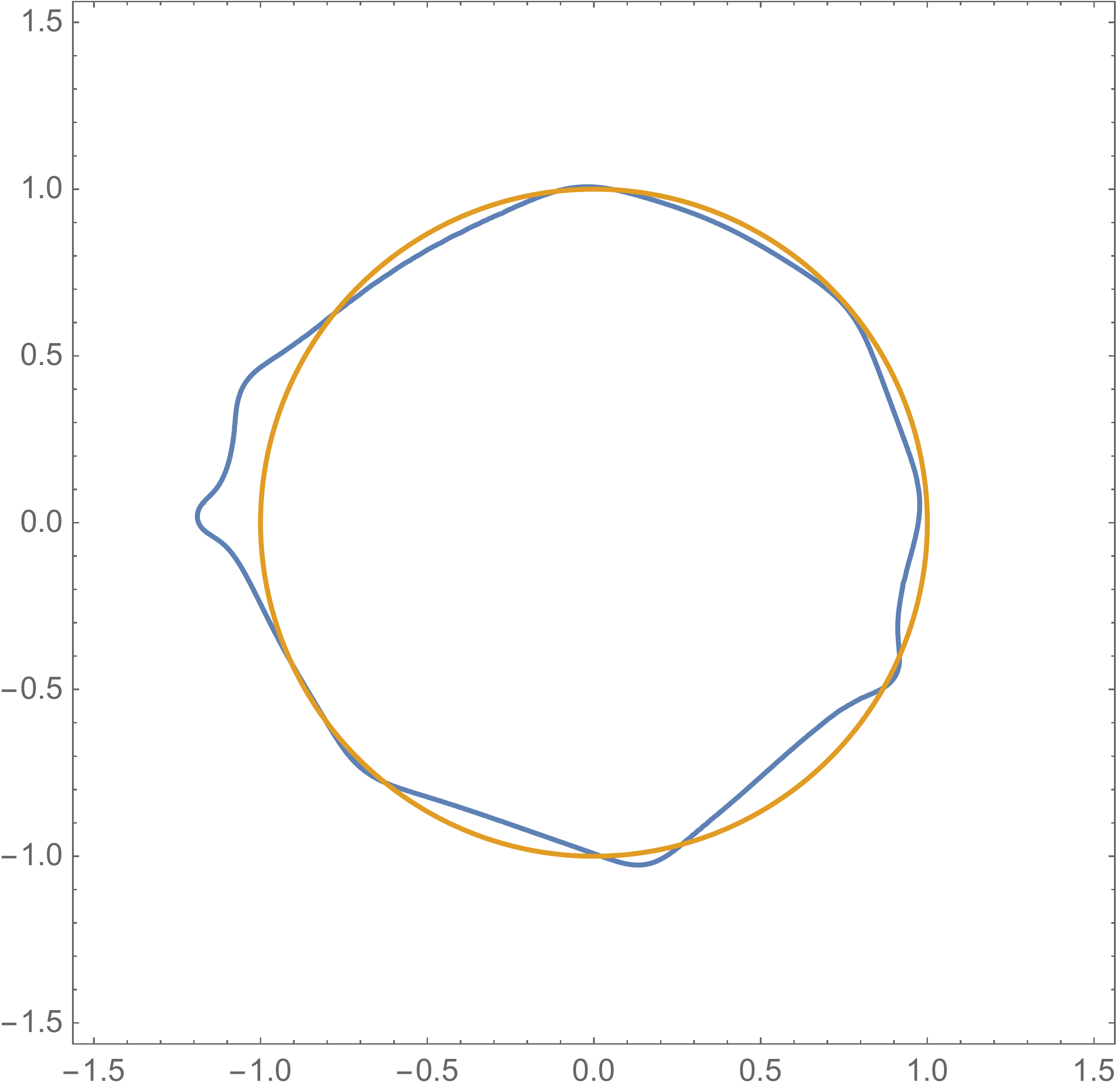}
\includegraphics[scale=0.17]{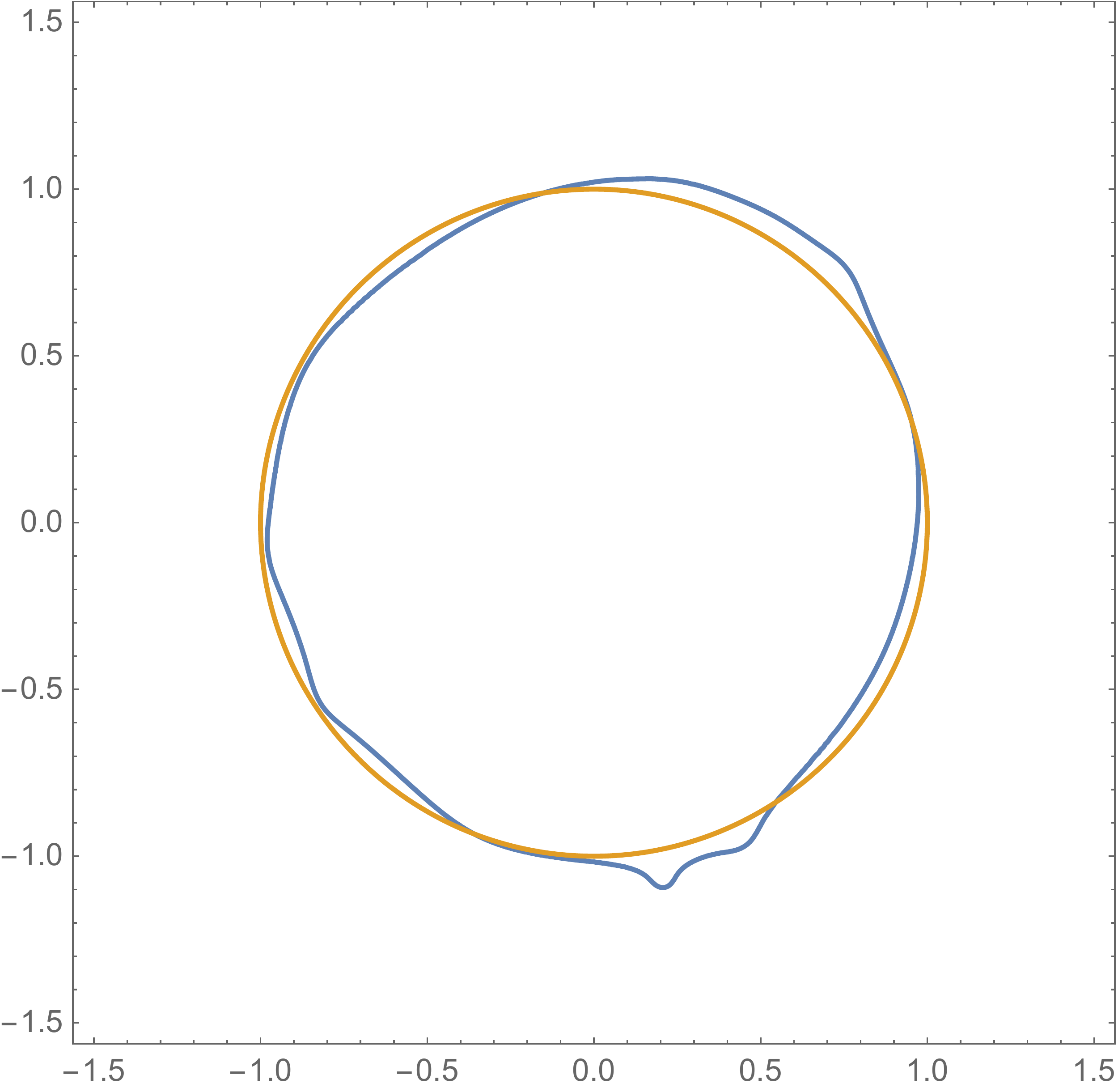}
\includegraphics[scale=0.17]{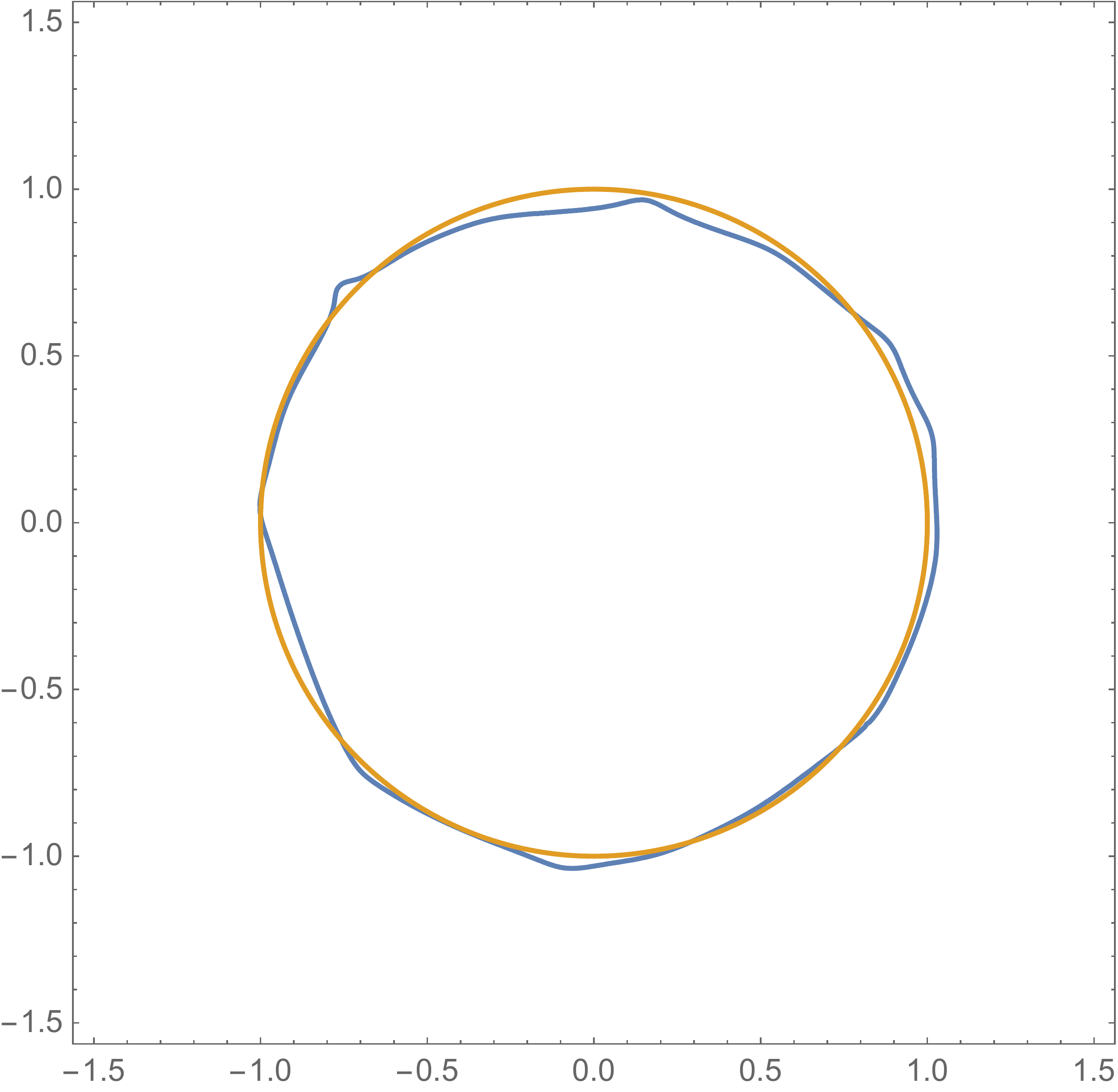}
\caption{Lemniscates of degree $n=20,30,40$ generated by the characteristic polynomial of a Ginibre matrix together with a unit circle plotted for reference.}
\label{fig:gin}
\end{figure}

%%% Do we need this figure at all?
%%%\begin{figure}[t]
%%%\centering
%%%\includegraphics[scale=0.17]{ellipse20.pdf}
%%%\includegraphics[scale=0.17]{ellipse30.pdf}
%%%\includegraphics[scale=0.17]{ellipse40.pdf}
%%%\caption{Lemniscates of degree $n=20,30,40$  with zeros sampled from the equilibrium measure of an ellipse.}
%%%\label{fig:ellipse}
%%%\end{figure}

%\begin{thm}\label{thm:unifcirc}
%Let $\mathbb{S}^1\subset\C$ denote the unit circle. Let $\{T_i\}$ be a sequence of random variables that are i.i.d. $\mbox{Unif}(\mathbb{S}^1)$.  Consider a  random polynomial  $\{p_n\}$ defined by
%$$p_n(z) = \prod_{k=1}^{n} (z-T_k),$$
%with $\Lambda_n$ being the corresponding lemniscate. Let
%$\D^{+}=\{z\in\D:\Im(z) > 0 \}.$ 
%Let $\theta\in [0, 2\pi)$ be the random direction defined by $\theta := \arg \sum_{i=1}^n T_i $.
%Fix a compact set $K \subset \D^+$.
%Then
%$$e^{i\theta} K \subset \Lambda_n  \quad\mbox{w.h.p}. $$
%In particular, 
%$$\rho_n\geq a \quad\mbox{w.h.p.,}\quad\forall{a\in (0, 1/2)}.$$
%\end{thm}

\subsection*{Outline of the paper}

We review some preliminary results in Section \ref{sec:prelim} that serve as tools in the proofs the results stated above. 
We prove Theorem \ref{G} and Corollary \ref{cor:convergenceofinradius} in Section \ref{sec:main}, and we prove Theorem \ref{thm:pospot} and Corollary \ref{cor:smalllemniscate} in Section \ref{sec:pospot}. The proof of Theorem \ref{thm:unifcirc}, concerning uniform measure on the circle, is presented in Section \ref{sec:unifcirc}, and Theorem \ref{Gin}, related to the Ginibre ensemble, is proved in Section \ref{sec:Ginibre}.

\section{Preliminary Results}\label{sec:prelim}
\noindent We start with two preparatory lemmas which we use repeatedly in the proofs of our theorems. 

\begin{lemma}\label{varbd}
Let $\mu$ be a Borel probability measure with compact support $S\subset\C$ satisfying Assumption $(A).$ Whenever $K$ is a non-empty compact subset of $\Omega^{-}$ or a compact subset of $\Omega^{+}$ with $K\cap S=\emptyset,$ there exists a constant $c(K) >0$ such that
$$\inf_{z\in K}\int_{S}\left(\log |z-w|\right)^2d\mu(w)\geq c(K). $$
\end{lemma}

\vspace{0.1in}

\begin{proof}
Let $\mu$ satisfy assumption \eqref{A}, and let $K\neq\emptyset$ be compact. Then, by the Cauchy-Schwarz inequality we have for all $z\in K$

\begin{equation}\label{lbvar}
\int_{S}\left(\log |z-w|\right)^2d\mu(w)\geq \left(\int_{S}\left|\log |z-w|\right|d\mu(w) \right)^2\geq |U_{\mu}(z)|^2\\
\end{equation}

Thus in order to prove the lemma, it suffices to show that $|U_{\mu}(z)|^2$ is bounded away from zero for $z\in K$, whenever $K\subset\Omega^{-},$ or $K\subset\Omega^{+}$ and $K\cap S = \emptyset.$ 

\vspace{0.1in}

Suppose first that $K\subset\Omega^{-}$ is compact. Since subharmonic functions are upper semi-continuous and hence attain a maximum on any compact set, there exists $c_1(K)> 0$ such that $U_{\mu}(z)\leq\  -c_1(K),$ for all $z\in K.$ Hence $|U_{\mu}(z)|^2\geq c_1(K)^2,$ for $z\in K.$ In the other case, let $K\subset \Omega^{+}$ be compact and disjoint from the support $S$ of $\mu.$ Notice then that $U_{\mu}(z)$ is positive and harmonic on $K.$ An application of Harnack's inequality now gives the existence of the required constant (depending only on $K$). This concludes the proof of the lemma.
\end{proof}

The second lemma is based on a net argument which allows us to control the size of the modulus of a polynomial by its values at the points of the net. 
\begin{lemma}\label{net}
Let $G$ be a bounded Jordan domain with rectifiable boundary. Let $p(z)$ be a polynomial of degree $n.$ Then, there exists a  constant $C= C(G)>0,$ and points $w_1, w_2...w_{Cn^2}\in\partial G$ such that
\begin{equation}\label{ubnd}
\|p\|_{\partial G}\leq 2\max_{1\leq k\leq Cn^2}|p(w_k)|
\end{equation}
\end{lemma}

\begin{proof}
The key to the proof is a Bernstein-type inequality (see \cite[Thm. 1]{PommDeriv})
\begin{equation}\label{eq:PommBern}
|p'(z)| \leq C_1 n^2 M,
\end{equation}
where $M := \| p \|_{\partial G}$, and
$C_1$ is a constant that depends only on $G$. With this estimate in hand, the proof reduces to the following argument that is well-known but which we nevertheless present in detail for the reader's convenience. Let $\ell = \ell(\partial G)$ denote the length of $\partial G.$ Let $N$ be a positive integer to be specified later. Divide $\partial G$ into $N$ pieces of equal length, with $w_0, w_1..., w_N$ denoting the points of subdivision. Let $z_0\in\partial G$ be such that $M= \|p\|_{\partial G}= |p(z_0)|.$ If $z_0$ is one of the $w_j,$ then the estimate \eqref{ubnd} clearly holds. If that is not the case, then $z_0$ lies  $|z_0-w_j|\leq \frac{\ell}{N}$, for some $j$ with $0\leq j\leq N$. We can now write 
\begin{equation}\label{Bernstein}
M - |p(w_j)|\leq |p(z_0) - p(w_j)|= \left|\int_{w_j}^{z_0}p'(t)dt \right|\leq C_1n^2M\frac{\ell}{N}.
\end{equation}
Here we have used the Bernstein-type inequality \eqref{eq:PommBern} to estimate the size of $|p'|$. If we now choose $N = 2\ell C_1n^2,$ then the estimate \eqref{Bernstein} becomes
$$M - |p(w_j)|\leq\dfrac{M}{2} $$
which concludes the proof of the lemma.
\end{proof}

We will also need the following concentration inequality (see Section~2.7 of \cite{BGM}).  This result, referred to as ``Bennett's inequality'', is similar to the well-known Hoeffding inequality, but note that, instead of being bounded, the random variables are merely assumed to be bounded from above.

\begin{thm}[Bennett's inequality]\label{Ben}
Let $X_1, X_2,..., X_n$ be independent random variables with finite variance such that $X_i\leq b$ for some $b>0$ almost surely for all $i\leq n.$ Let
$$S = \sum_{i=1}^{n} (X_i - \E(X_i))$$
and $\nu = \sum_{i=1}^{n}\E(X_i^2).$ Then for any $t>0,$
$$\mathbb{P}(S > t)\leq\exp\left(\dfrac{-\nu}{b^2}h\left(\dfrac{bt}{\nu}\right)\right),$$
where $h(u) = (1+u)\log(1+u) - u$ for $u> 0.$
\end{thm}

\section{Proofs of Theorem \ref{G} and Corollary~\ref{cor:convergenceofinradius}}\label{sec:main}
\begin{proof}[Proof of Theorem \ref{G}]
We divide the proof into two steps.

\vspace{0.1in}

\noindent \textbf{Step $1$: Compact subsets of $\Omega^{-}$ lie in $\Lambda_n$}

\vspace{0.1in}

\noindent By our hypothesis $\Omega^{-}\neq\emptyset.$ Let $K\subset\Omega^{-}$ be compact. We wish to show that $K\subset\Lambda_n$ w.o.p. We may assume without loss of generality that $K=\overline{G}$ for some bounded Jordan domain $G$ with rectifiable boundary, since any connected compact is contained such a domain.
Recall that $\Lambda_n = \{ z: \log |p_n(z)| < 0 \}$. Writing $$\log|p_n(z)| = \sum_{k=1}^n \log|z-X_k|$$
as a sum of i.i.d. random variables, for $z\in\Omega^{-}$ we will use a concentration inequality to show that $\log |p_n(z)|$ is negative with overwhelming probability. We then use lemma \ref{net} to get a uniform estimate on $K$ and finish the proof. 

Fix $z_0\in K.$ For $k=1,2,...,n,$ define $Y_k = \log|z_0-X_k|$, and let
$$Z := \sum_{k=1}^n ( Y_k - \E Y_k).$$
Notice since $z_0\in\Omega^{-}$
$$\E Y_k = U_{\mu}(z_0) < 0$$
and by the assumption in the statement of the theorem, we also have 
\begin{align*}
    \sigma_{z_0}^2:=\E Y_k^2 &= \int_{S} (\log|z_0 - u|)^2 d\mu(u) < \infty \\
\end{align*}
\vspace{0.1in}

\noindent Now applying Theorem \ref{Ben} to our problem with $b\geq\sup_{z\in K, w\in S}\log(|z|+ |w|) $, $\nu= n\sigma_{z_0}^2,$ we obtain
\begin{align*}\label{eqp}
 \mathbb{P}\{\log|p_n(z_0)| > -\log(2)\}&= \mathbb{P}\{\log|p_n(z_0)|-nU_{\mu}(z_0) > -nU_{\mu}(z_0) -\log(2)\}\\
 &=\mathbb{P}\{Z > -nU_{\mu}(z_0) -\log(2)\}\\
 &\leq \exp\left(-\dfrac{n\sigma_{z_0}^2}{b^2}h\left(\dfrac{-b}{\sigma_{z_0}^2}U_{\mu}(z_0) - \dfrac{b\log(2)}{n\sigma_{z_0}^2} \right) \right).
\end{align*}
Since subharmonic functions are upper semi-continuous and hence attain a maximum on any compact set, we have, $U_{\mu}(z)\leq - M$ for all $z\in K$ and some $M> 0.$ Also, by Lemma~\ref{varbd}, $0< c_1(K)\leq\sigma_{z}^2\leq c_2(K) <\infty,$ for all $z\in K.$ This bound together with the fact that $h$ is an increasing function can now be used in the above estimate to get

\begin{equation}\label{eq4}
\mathbb{P}\{\log|p_n(z_0)| > -\log(2)\}\leq\exp\left(-\dfrac{n\sigma_{z_0}^2}{b}h\left(\dfrac{-b}{\sigma_{z_0}^2}U_{\mu}(z_0) - \dfrac{b\log(2)}{n\sigma_{z_0}^2} \right) \right)\leq \exp\left(-cn \right)   
\end{equation}

\noindent for some constant $c= c(K) >0$ depending only on $K.$ Using lemma \ref{net} in combination with a union bound and the estimate \eqref{eq4}, we obtain
\begin{align*}
\PP\{\log\|p_n\|_K < 0 \}&\geq \PP\{\max_{1\leq k\leq Cn^2}{\log|p_n(w_{k,n})|} + \log(2)< 0 \} \\
&= 1 - \PP\{\max_{1\leq k\leq Cn^2}{\log|p_n(w_{k,n})|}> - \log(2)\} \\
&= 1- \PP\left(\bigcup_{k=1}^{Cn^2}\{\log|p_n(w_{k,n})|> - \log(2)\}\right) \\
&\geq 1-Cn^2\exp\left(-cn \right)
%\to 1 \quad\mbox{as}\quad n\to\infty.
\end{align*}
where in the last inequality we used \eqref{eq4}. This proves that $K\subset\Lambda_n$ w.o.p. and concludes the proof of the first part. 

\vspace{0.1in}

\noindent \textbf{Step 2: Compact subsets $L$ of $\Omega^{+}\setminus S$ are in $ \Lambda_n^c$}.

\vspace{0.1in}

\noindent Without loss of generality, we may assume that $L$ is a closed disc in $\Omega^{+}\setminus S$. Since $S$ is a compact set disjoint from $L$, there exists $\delta >0$ such that the distance $d(L, S) = \delta.$  Notice that for all $z\in L,$ we have $-\log|z - X_i|\leq -\log\delta$.  Now fix $z_0\in L.$ An application of Bennett's inequality to the random variables $- \log|z_0 - X_i|$ yields,
\begin{equation}\label{b2}
\PP\left(-\log|p_n(z_0)| + nU_{\mu}(z_0)\geq nU_{\mu}(z_0) -1 \right)\leq \exp\left(-\dfrac{n\sigma_{z_0}^2}{b}h\left(\dfrac{b}{\sigma_{z_0}^2}U_{\mu}(z_0)-\dfrac{b}{n\sigma_{z_0}^2}) \right) \right).
\end{equation}
The quantities $b, h$ have an analogous meaning as in Step $1$. By Lemma \ref{varbd}, $\sigma_z^2$ is bounded below, and by assumption it is also bounded above, by some positive constants depending only on $L$. Furthermore, Lemma \ref{varbd} shows that $U_{\mu}(z)\geq c(L) >0$ for all $z\in L.$ Making use of all this in \eqref{b2}, we can now estimate
\begin{align}
\PP\left(\log|p_n(z_0)| > 1\right)&=\PP\left(\log|p_n(z_0)| - nU_{\mu}(z_0)> -nU_{\mu}(z_0)+1 \right) \nonumber \\
&= 1 - \PP\left(\log|p_n(z_0)| - nU_{\mu}(z_0)\leq -nU_{\mu}(z_0)+1 \right) \nonumber  \\
&= 1 - \PP\left(-\log|p_n(z_0)| + nU_{\mu}(z_0)\geq nU_{\mu}(z_0) -1\right) \nonumber  \\
&\geq 1 - \exp\left(-\dfrac{n\sigma_{z_0}^2}{b}h\left(\dfrac{bU_{\mu}(z_0)}{\sigma_{z_0}^2 }-\dfrac{b}{n\sigma_{z_0}^2} \right)\right) \nonumber  \\
&\geq 1 - \exp\left(-C_0(L)n \right). \label{Ben2}
%&\to 1\quad\mbox{as}\quad n\to\infty.
\end{align}
%\commE{maybe clarify $C_0$ before next sentence.}
%We have thus shown that there exists a constant $C_0>0$ depending only on $L$ such that
%\be\label{Ben2}
%\PP \left( \log |p_n(z_0| > 1 \right)  \geq 1- \exp\left( - C_0 \, n \right).
%\ee
This estimate shows that individual points of $L$ are in $\Lambda_n ^c$ with overwhelming probability. To finish the proof, we once again use a net argument to show that $L\subset\Lambda_n^c$ w.o.p. We first observe that if $z, w\in L$, and $X$ is one of the $X_k$'s, the mean value theorem gives
$$|\log|z - X| - \log|w - X|| \leq \dfrac{|z - w|}{\delta},$$   
\noindent where we have used that $d(L, S) =\delta >0$ (and that $L$ is a disk). The triangle inequality then yields \be\label{eq:triangle}
|\log|p_n(z)| - \log|p_n(w)| | \leq \dfrac{n|z-w|}{\delta}, \quad \text{for } z,w \in L.
\ee
Choose a net of $n^2$ equally spaced points $w_1, w_2,... w_{n^2}$ on $\partial L$, and note that any point on $\partial L$ is within $C_1/n^2$ of some point in the net, where $C_1$ is a constant depending on the radius of $L$.  From \eqref{eq:triangle} we have that
\begin{equation}\label{net2}
|\log|p_n(z)| - \log|p_n(w)| |\leq \frac{C_2}{n}, \quad \text{for } z,w \in L \text{ with } |z-w|\leq\frac{C_1}{n^2}, \end{equation}
where $C_2 = C_1/\delta$ is a constant.

\noindent   We are now ready to show that for large $n$, $\inf_{z\in L}\log|p_n(z)| > 0$ w.o.p. Indeed, note that the point on $\partial L$ where the infimum of $\log|p_n|$ is attained must be within $C_1/n^2$ of some point in the net $\{w_1, w_2..., w_{n^2}\}$. Then by \eqref{net2},
$$\PP\left(\inf_{L}\log|p_n(z)| > 0\right)\geq \PP\left(\bigcap_{k=1}^{n^2}\{\log|p_n(w_k)|> 1\}\right).$$

\noindent Therefore, we obtain
\begin{align*}
\PP\left(\inf_{L}\log|p_n(z)| > 0\right)&\geq \PP\left(\bigcap_{k=1}^{n^2}\{\log|p_n(w_k)|> 1\}\right)\\
&= 1 - \sum_{k=1}^{n^2}\PP\left(\log|p_n(w_k)|\leq 1 \right)\\
&\geq 1 - n^2\exp(-C_0 \, n).
%&\to 1 \quad\mbox{as}\quad n\to\infty.
\end{align*}
by the pointwise estimate \eqref{Ben2}. This concludes the proof of the theorem.
\end{proof}

%\subsection{Proof that $\rho_n\to \rho$ when $S\subseteq \Omega^{-}$}
\begin{proof}[Proof of Corollary~\ref{cor:convergenceofinradius}]
We assume that the measure $\mu$ is as in Theorem~\ref{G}.  Let $\rho_n=\rho(\Lambda_n)$ be the inradius of the lemniscate of $p_n$ and let $\rho=\rho(\Omega^{-})$ be the inradius of $\Omega^{-}$. By Theorem~\ref{G}, we immediately get $\liminf \rho_n\ge \rho$. 

Let $S$ be the support of $\mu$. As $S\cap \Omega^{+}=\emptyset$,  Theorem~\ref{thm:pospot} shows that if $m>0$ then $\Lambda_n\cap \{U_{\mu}\ge m\}$ is contained in a union of at most $n$ circles each of radius $e^{-cn}$. Writing $\rho_n(m)$ for $\rho(\Lambda_n\cap \{U_{\mu}<m\})$ and $\rho(m)$ for $\rho(\{U_{\mu}<m\})$, it is then clear that  $\rho_n \le \rho_n(m)+2ne^{-cn}\le \rho(m)+2ne^{-cn}$ and therefore, first letting $n\to \infty$ and then letting $m\downarrow 0$ we see that  $$\limsup_{n\to \infty} \rho_n\le \lim_{m\downarrow 0}\rho(m).$$

As $U_{\mu}$ is continuous on $\C \setminus S$, it follows that for any $\epsilon>0$ there is $m>0$ such that $\{U_{\mu}<m\}\subseteq \Omega_{\epsilon}^{-}$, the $\epsilon$ enlargement of $\Omega^{-}$. Hence, with $\rho'(\epsilon):=\rho(\Omega_{\epsilon}^{-})$, we have
$$\limsup_{n\to \infty} \rho_n\le \lim_{\epsilon\downarrow 0}\rho'(\epsilon).$$
Under the additional assumption that $S\subseteq \Omega^{-}$, we show that $\rho'(\epsilon)\downarrow \rho$ as $\epsilon \downarrow 0$ and that completes the proof that $\limsup \rho_n\le \rho$. That $\rho'(\epsilon)\downarrow \rho$ requires a proof as  inradius is not continuous under decreasing limits of sets. For example, the inradius of the slit disk $\D\setminus [0,1)$ is $1/2$ but any $\epsilon$-enlargement of it has inradius $1$.

As $U_{\mu}$ is harmonic on $\C\setminus S$ and $S\subseteq \Omega^{-}$ and $U_{\mu}(z)\sim \log |z|$ near $\infty$,  the level set $\{U_{\mu}=0\}$ is a compact set comprised of curves that are real analytic except for a discrete set of points (the critical points of $U_{\mu}$ are  zeros of  locally defined analytic functions). It also separates $S$ from $\infty$. 
Thus, $\{U_{\mu}<0\}$ can be written as a union of Jordan domains, and there are at most finitely many components that have inradius more than any given number. 

Pick a component $V$ of $\Omega^{-}$ that attains the inradius $\rho$. The boundary of $V$ can have a finite number of critical points of $U_{\mu}$. Locally around any such critical point, $U_{\mu}$ is the real part of a holomorphic function that looks like $cz^p$ for some $p$, and hence $U_{\mu}=0$ is like a system of equi-angular lines with angle $\pi/p$ between successive rays. In particular, there are no cusps. What this shows is that $V$ satisfies the following ``external ball condition": There is a $\delta_0>0$ and $B<\infty$, such that for any $\delta<\delta_0$ and each $w\in \partial V$, there is a 
\begin{equation}\label{eq:exteriorballcondition}
w'\in \C\setminus V \; \mbox{ such that }|w'-w|=\delta \mbox{ and }|w'-z|\ge \delta/B \mbox{ for all }z\in \Omega^{-}. 
\end{equation}
Now suppose $\D(z,r)\subseteq \Omega_{\epsilon}^{-}$. If $\epsilon<\delta_0/B$, we claim that $\D(z,r-2B\epsilon)\subseteq \Omega^{-}$, which of course proves that $\rho \ge \rho'(\epsilon)-B\epsilon$, completing the proof. If the claim was not true, then we could find $w\in \partial V$ such that $|w-z|\le r-2B\epsilon$. Find $w'$ as in \eqref{eq:exteriorballcondition} with $\delta=B\epsilon$. Then $w'\not\in \Omega_{\delta/B}=\Omega_{\epsilon}$ but 
$$
|w'-z|\le |w'-w|+|w-z|\le \frac{\delta}{B}+r-2B\epsilon <r.
$$
This is a contradiction as $w'\in \D(z,r)\subseteq \Omega_{\epsilon}$.
\end{proof}

\section{Proof of Theorem \ref{thm:pospot} and Corollary~\ref{cor:smalllemniscate}}\label{sec:pospot}
A standard net argument can be used to prove the theorem. But we would like to first present a proof of Corollary~\ref{cor:smalllemniscate} by a different method, which may be of independent interest. At the end of the section, we outline  the net argument  to prove Theorem~\ref{thm:pospot}. 

We will need the following lemma in the proof of Corollary~\ref{cor:smalllemniscate}.
\begin{lemma}\label{lemma:logP'}
Under the  assumptions of Corollary~\ref{cor:smalllemniscate}, there exists $c_1>0$ such that
\be
\PP \left( \log|p_n'(X_1)| \leq \frac{m}{2}(n-1) \right) \leq e^{-c_1 n}.
\ee
\end{lemma}
First we prove the corollary assuming the above Lemma.
\begin{proof}[Proof of Corollary~\ref{cor:smalllemniscate}]
Let $G_i$ be the connected component of $\Lambda_n$ containing $X_i$.
Then by Bernstein's inequality we have
\be\label{eq:Bern}
|p_n'(X_i)| \leq C\frac{n^2}{\text{diam}(G_i)} \| p_n \|_{\partial G_i} = C \frac{n^2}{\diam(G_i)}.
\ee
By Lemma \ref{lemma:logP'} we have
$$|p_n'(X_i)| \geq \exp \left\{ \frac{m}{2}(n-1) \right\}, \quad \text{w.o.p.}$$
and we conclude from \eqref{eq:Bern} that
\be\label{eq:diam}
\diam(G_i) \leq C n^2 \exp \left\{ -\frac{m}{2}(n-1) \right\}, \quad \text{w.o.p.}
\ee
The event $\Lambda_n \subset \bigcup_{k=1}^n \D_{r_n}(X_k)$
occurs if $\diam (G_i) < r_n$ for each $i=1,2,...,n$. Using \eqref{eq:diam} and a union bound, all these events occur with  overwhelming probability if we choose $r_n = \exp\{-c_0 n\}$ for a suitable $c_0$.
\end{proof}
It remains to prove Lemma \ref{lemma:logP'}.
\begin{proof}[Proof of Lemma \ref{lemma:logP'}]
We have
\begin{align}\label{eq:beginlemma}
\PP\left\{ \log|p_n'(X_1)| \leq \frac{m}{2}(n-1) \right\} &= \int_S  \PP \left\{ \log|p_n'(X_1)| \leq \frac{m}{2}(n-1) \big\rvert X_1 = z\right\} d\mu(z) \\
&= \int_S \underbrace{\PP \left\{ \sum_{k=2}^n \log|z-X_k| \leq \frac{m}{2}(n-1) \right\}}_{(*)} d\mu(z).
\end{align}
Let us rewrite the integrand $(*)$ as
$$(*) = \PP \left\{ Z \geq \left( U_\mu(z) - \frac{m}{2}\right)(n-1) \right\} , \quad Z=(n-1)U_\mu(z) - \sum_{k=2}^n \log|z-X_k|.$$
Then we have (with $\theta$ to be chosen below)
\begin{align}
(*) &= \PP\left\{ e^{\theta Z} \geq e^{\theta(n-1)(U_\mu - m/2)} \right\} \nonumber \\
 \text{(since $U_\mu \geq m$)}\quad  &\leq \PP\left\{ e^{\theta Z} \geq e^{\theta(n-1)(m/2)} \right\} \nonumber \\
 &\leq e^{-\theta (n-1)(m/2)} \E e^{\theta Z}. \label{eq:mgfboundforprob}
\end{align}

Let $Z_k = -\log|z-X_k|+U_\mu(z)$ so that $Z=Z_2+\ldots +Z_n$. As $X_i$ are i.i.d., so are $Z_i$ and we have
\begin{align*}
    \EE e^{\theta Z} &= \left( \EE e^{\theta Z_2}\right)^{n-1}.
\end{align*}
We claim that there exist $\tau<\infty$ and $\theta_0>0$ (not depending on $z\in S$) such that  
\begin{align}\label{claim:subexp}
\EE[e^{\theta Z_2}]\le e^{\tau\theta^2} \;\;\; \mbox{ for }|\theta|<\theta_0.
\end{align} 
Assuming this, the proof can be completed as follows:
%%The assumption (B) implies that $\EE e^{\theta Z_1} < \infty$.
%%Indeed,
%%\begin{align*}
%%    \EE e^{\theta Z_1} &= 
%%    e^{\theta U_\mu(z)}\EE e^{-\theta \log| 
%%    z-X_1|
%%    } \\
%%    &\leq \EE e^{\theta M_1} \int_S \frac{d \mu(\omega)}{|z-\omega|^\theta}
%%\end{align*}
%%where $M_1 = \sup_{z \in S} U_\mu(z)$.
%%By assumption~(B), the integral if finite for $\theta< \delta$. Thus $\EE e^{\theta Z_1} < \infty$ for $\theta<\delta$ and hence $Z_1$ has finite moments. Denoting $\tau_z = \EE Z_1^2$, we have
%%\begin{align*}
%%    \EE e^{\theta Z_1} &= \EE \left( 1+\theta Z_1 +\frac{\theta^2}{2} Z_1^2 + o(1) \right)
%%    \\
%%    &= 1+\frac{\theta^2}{2} \tau_z + o(1), \quad \text{as } \theta \rightarrow 0,
%%\end{align*}
%%where we used $\EE Z_1 = 0$. Note that by lemma \ref{varbd}, $\tau_z < M_2$ for all $z \in K$ for some constant $M_2>0$. Plugging this into \eqref{eq:mgfboundforprob}, we get {\color{red} Uniformity of $o(1)$ needed?}
\begin{align}
%\PP \left\{ e^{\theta Z}  \geq e^{\theta(n-1)m/2} \right\}
% &\leq e^{-\theta(n-1)m/2} \EE (e^{\theta Z}) \\
(*)&\leq e^{-\theta(n-1)m/2}  e^{(n-1)\tau \theta^2} \nonumber \\
&= e^{ -\frac14 m \theta(n-1)} \label{eq:subexpbound}
\end{align}
provided we choose $\theta<\frac{m}{4\tau}$. Using this in \eqref{eq:beginlemma} we obtain
\be 
\PP\left\{ \log|p_n'(X_1)| \leq \frac{m}{2}(n-1) \right\} 
\leq e^{ -\frac14 m \theta(n-1)},\ee
which implies the statement in the lemma.

It remains to prove \eqref{claim:subexp}. Assumption (D) in definition~\ref{A} yields that for $z\in S$,
\begin{align*}
\PP\{Z_1>t\}&=\PP\{|z-X_1|\le e^{U_{\mu}(z)-t}\} \\
&\le Ce^{\epsilon(M-t)}
\end{align*}
where $M=\sup_{z\in S}U_{\mu}(z)$. On the other hand, $\PP\{Z_1<-t\}=0$ for large $t$, hence by choosing a smaller $\epsilon$ if necessary, we have the bound
$$
\PP\{|Z_1|>t\}\le 2e^{-\epsilon t}.
$$
A random variable satisfying the above tail bound is said to be {\em sub-exponential} (see Section~2.7 in \cite{vershynin}). It is well-known (see the implication $(a)\implies (e)$ of Proposition~2.7.1 in \cite{vershynin}) that if a sub-exponential random variable has zero  mean, then  \eqref{claim:subexp} holds.
\end{proof}

Now we outline the argument for the proof of Theorem~\ref{thm:pospot}
\begin{proof}[Proof of Theorem~\ref{thm:pospot}]
The same argument (basically that $-\log|z-X_1|+U_{\mu}(z)$ has sub-exponential distribution) that led to \eqref{eq:subexpbound} shows that there exists $\theta>0$ 
\begin{align}\label{eq:ubdlogpn}
\PP\{\log|p_n(z)|<\frac12 m n \}\le e^{-\theta n}
\end{align}
for any $z\in L$. Let $r_n=e^{-\frac{\theta}{4} n}$. Then, if $z\in L\setminus \bigcup_{k=1}^nB(X_k,r_n)$, we have
$$
|\nabla \log |p_n(z)||=\big| \sum_{k=1}^n\frac{1}{z-X_k}\big| \; \le \; \frac{n}{r_n}.
$$
Therefore, if $z\in L\setminus \bigcup_{k=1}^nB(X_k,(1+\frac{m}{4})r_n)$, then combining the bound on the gradient  with \eqref{eq:ubdlogpn},  we get
$$
\PP\left\{\inf_{B(z,\frac14 mr_n)}\log|p_n|\ge \frac14 mn\right\}\ge 1-e^{-\theta n}.
$$
Assuming without loss of generality that $m\le 1$, we may choose a net of $C/r_n^2$ points in $L$ such that every of point of $L\setminus \bigcup_{k=1}^nB(X_k,2r_n)$  is within distance $mr_n/4$ of one of the points of the net. Then, $\log|p_n|>\frac14 mn$ everywhere on $L\setminus \bigcup_{k=1}^nB(X_k,2r_n)$, with probability at least $1-\frac{C}{r_n^2}e^{-\theta n}\ge 1-Ce^{-\frac{\theta}{2}n}$, by our choice of $r_n$.
\end{proof}

\section{Proof of Theorem~\ref{thm:unifcirc}}\label{sec:unifcirc}
First we claim that   $\Lambda_n\subseteq (1+\epsilon)\D$ w.o.p. for any $\epsilon>0$. Deterministically, $\Lambda_n\subseteq 2\D$, since $\mu$ is supported on $\mathbb{S}^1$. Further, $U_{\mu}(z)=\log_+ |z|$, hence $L=\{z\ : \ 1+\epsilon\le |z|\le 2\}$ is a compact subset of $\Omega^{+}$. By Theorem~\ref{G} or Theorem~\ref{thm:pospot}, we see that $L\cap \Lambda_n=\emptyset$ w.o.p.  proving that $\Lambda_n\subseteq (1+\epsilon)\D$ w.o.p.

Thus, it suffices to consider $\Lambda_n\cap \D$. Consider
$$
g_n(z)=\frac{1}{\sqrt{n}}\sum_{k=1}^n\log|z-X_k|
$$
for $z\in \D$. As $X_k$ are uniform on $\mathbb{S}^1$, it follows that $\E[\log|z-X_1|]=0$.  Let 
$$
K(z,w)=\E[(\log |z-X_1|)(\log |w-X_1|)]=\frac{1}{2\pi}\int_{0}^{2\pi}\log|z-e^{i\theta}| \log|w-e^{i\theta}| \ d\theta.
$$
Hence $\E[g_n(z)]=0$ and $\E[g_n(z)g_n(w)]=K(z,w)$.

Let $g$ be the (real-valued) Gaussian process on $\D$ with expectation $\E[g(z)]=0$ and covariance function $\E[g(z)g(w)]=K(z,w)$. Then by the central limit theorem, it follows that
$$
(g_n(z_1),\ldots ,g_n(z_k))\stackrel{d}{\rightarrow} (g(z_1),\ldots ,g(z_k))
$$
for any $z_1,\ldots ,z_k\in \D$. We observe that $g_n(0)=0$ and claim that $\sup_{r\D} |\nabla g_n|$ is tight, for any $r<1$. By a well-known criterion for tightness of measures (on the space $C(\D)$ endowed with the topology of uniform convergence on compacts),  this proves that $g_n\to g$ in distribution, as processes (see Theorem~7.2 in \cite{billingsley}).

To prove the  tightness of $\sup_{r\D} |\nabla g_n|$, fix $r<s<1$ and note that $\nabla g_n(z)$ is essentially the same as $F_n(z)=\frac{1}{\sqrt{n}}\sum_{k=1}^n\frac{1}{z-X_k}$ which is holomorphic on $\D$. By Cauchy's integral formula, for $|z|<r$, 
\begin{align*}
|F_n(z)|^2 &=\big|\frac{1}{2\pi}\int_0^{2\pi}\frac{F_n(se^{i\theta})}{z-se^{i\theta}}ise^{i\theta}d\theta\big|^2 \\
&\le \left(\frac{1}{2\pi}\int_{0}^{2\pi}|F_n(se^{i\theta})|^2d\theta\right)\left(\frac{1}{2\pi}\int_0^{2\pi}\frac{1}{|z-se^{i\theta}|^2}d\theta\right) \\
&\le \frac{1}{(s-r)^2}\frac{1}{2\pi}\int_{0}^{2\pi}|F_n(se^{i\theta})|^2d\theta.
\end{align*}
The bound does not depend on $z$, hence taking expectations,
\begin{align*}
\E[(\sup_{r\D}|F_n|)^2]&\le \frac{1}{(s-r)^2}\frac{1}{2\pi}\int_{0}^{2\pi}\E\left[|F_n(se^{i\theta})|^2\right]d\theta \\ 
&\le  \frac{1}{(s-r)^2}\frac{1}{2\pi}\int_{0}^{2\pi}\E\left[\frac{1}{|se^{i\theta}-X_1|^{2}}\right]d\theta \\
& \le  \frac{1}{(s-r)^2(1-s)^2}.
\end{align*}
The boundedness in $L^2$ implies tightness of the distributions of $F_n$, as claimed.

In order to formulate a precise statement on almost sure convergence it is necessary to
construct $g_n$ and $g$ on a single probability space.
One way to accomplish that is by the Skorokhod representation theorem (see Theorem~6.7 in \cite{billingsley}) from which it follows that $g_n$ and $g$ can be constructed on one probability space so that $g_n\to g$ uniformly on compacta, a.s. Hence, the proof of Theorem~\ref{thm:unifcirc} will be complete if we prove the following lemma.

\begin{lemma}\label{lem:inradiusconvergence} Let $f_n,f;\D\to \R$ be smooth functions such that $\{f=0\}\cap \{\nabla f=0\}=\emptyset$. Suppose $f_n\to f$ uniformly on compact sets of $\D$. Then, $\rho(\{f_n<0\})\to \rho(\{f<0\})$.
\end{lemma}
Indeed, applying this to $g_n,g$, we see that $\rho(\Lambda_n\cap \D)\to \rho(\{g<0\})$ almost surely. On the other hand, for any $\epsilon>0$, Theorem~\ref{thm:pospot} shows that  $\Lambda_n\cap ((1+\epsilon)\D)^c$ is contained in a union of $n$ disks of radius $e^{-cn}$, w.o.p. Putting these together, $\rho(\Lambda_n)\to \rho(\{g<0\})$ a.s. and hence in distribution. This  completes the proof of the convergence claim in Theorem~\ref{thm:unifcirc}.

\begin{proof}[Proof of Lemma~\ref{lem:inradiusconvergence}]
For any $U\subseteq \D$, it is clear that $\rho(U)-\epsilon\le \rho(U\cap (1-\epsilon)\D)\le \rho(U)$. Applying this to $U=\{f_n<0\}$ and $U=\{f<0\}$, we see that to show that  $\rho(\{f_n<0\})\to \rho(\{f<0\})$, it is sufficient to show that $\rho(\{f_n<0\}\cap (1-\epsilon)\D)\to \rho(\{f<0\}\cap (1-\epsilon)\D)$ for every $\epsilon>0$. On $(1-\epsilon)\D$, the convergence is uniform, hence for any $\delta>0$, we have $\{f<-\delta\}\subseteq \{f_n<0\}\subseteq \{f<\delta\}$ for sufficiently large $n$. It remains to show that $\delta\mapsto \rho(\{f<\delta\})$ is continuous at $\delta=0$.

First we show  that $\rho(\{f<-\delta\})\uparrow \rho(\{f<0\})$ as $\delta\downarrow 0$. If $B(z,r)\subseteq \{f<0\}$, then for any $\epsilon>0$, the maximum of $f$ on $B(z,r-\epsilon)$ is some $-\delta<0$. Hence $\rho(\{f\le -\delta\})\ge r-\epsilon$ proving that $\rho(\{f<-\delta\})\uparrow \rho(\{f<0\})$.

Next we show that $\rho(\{f\le \delta_n\})\downarrow \rho(\{f\le 0\})$ for some $\delta_n\downarrow 0$. Let $r_n= \rho(\{f\le \frac{1}{n}\})$ and find $z_n$ such that $B(z_n,r_n)\subseteq \{f\le \frac{1}{n}\}$. Let $r_n\downarrow r_0$ and $z_n\to z_0$ without loss of generality. Then if $w\in B(z_0,r_0)$, then $w\in B(z_n,r_n)$ for large enough $n$, hence $f(w)\le \frac{1}{n}$ for large $n$. Thus $f\le 0$ on $B(z_0,r_0)$ showing that $\rho(\{f\le 0\})\ge \lim\limits_{\delta\downarrow 0}\rho(\{f\le \delta\})$.

From the assumption that $\{f=0\}\cap \{\nabla f=0\}=\emptyset$, we claim that  $\rho(\{f\le 0\})=\rho(\{f<0\})$. Indeed, if $B(z,r)\subseteq \{f\le 0\}$, then in fact $B(z,r)\subseteq \{f< 0\}$. Otherwise, we would get $w\in B(z,r)$ with $f(w)=0$ which implies that $w$ is a local maximum of $f$ and hence $\nabla f(w)=0$.

This proves the continuity of $\delta\mapsto \rho(\{f<\delta\})$ at $\delta=0$, and  hence the lemma.
\end{proof}

This completes the proof of the first part that $\rho_n = \rho(\{g_n<0\})$ converges in distribution to
$\rho=\rho(\{g<0\})$. 
To show that $\mathbb P(\{\rho<\epsilon\})>0$, it suffices to show that $g>0$ on $(1-\epsilon)\D\cap \{|\Im z|>\epsilon\}$ with positive probability. To show that $\mathbb P(\{\rho>\frac12-\epsilon\})>0$, it suffices to show that $g<0$ in  $(1-\epsilon)\D\cap \{|\Im z|>\epsilon\}$ with positive probability. We do this in two steps.
\begin{enumerate}
\item There exist $u_0:\D\to \R$, harmonic with $u_0(0)=0$ such that $u_0<0$ on $(1-\epsilon)\D\cap \{|\Im z|>\epsilon\}$. This is known, see either the proof of Theorem~6.1 of \cite{NSP} or take $\log|p|$ of the polynomial $p$ constructed in Lemma~5 of  Wagner~\cite{wagner}.
\item For any $u:\D\to \R$, harmonic with $u(0)=0$ and any $r<1$ and $\epsilon>0$, we claim that $\|g-u\|_{\sup(r\D)}<\epsilon$ with positive probability. Applying this to  $u_0$ and $-u_0$ from the previous step show that $\rho>\frac12-\epsilon$ with positive probability and $\rho<\epsilon$  with positive probability.

To this end, we observe that the process $g$ can be represented as 
$$
g(z)=\Re\sum_{k=1}^{\infty} \frac{2}{k}a_k z^k
$$
where $a_k$ are i.i.d. standard complex Gaussian random variables. The covariance of $g$ defined as above is
$$
\E[g(z)g(w)]= \sum_{k\ge 1}\frac{1}{k^2}(z^k\bar{w}^k+w^k\bar{z}^k)
$$
which can be checked to match with the integral expression for $K(z,w)$ given earlier. Given any harmonic $u:\D\to \R$ with $u(0)=0$, write it as
$$
u(z)=\Re \sum_{k\ge 1}c_kz^k
$$
and  choose $N$ such that 
$$
\|\sum_{k>N}c_kz^k\|_{\sup(r\D)}<\epsilon.
$$
If both the events
\begin{align*}
\mathcal A_N&=\left\{\|\sum_{k>N}\frac{a_k}{k}z^k\|_{\sup(r\D)}<\epsilon\right\}, \\
\mathcal B_N&=\left\{|\frac{2a_k}{k}-c_k|<\frac{\epsilon}{N} \mbox{ for }1\le k\le N\right\}
\end{align*}
 occur, then $|g-u|<3\epsilon$ on $r\D$. As $\mathcal A_N$ and $\mathcal B_N$ are independent and have positive probability, we also have $\PP(\mathcal A_N\cap \mathcal B_N)>0$. 

\end{enumerate}

%\subsection{Gaussian coefficients}
%Suppose we  change the model by letting the randomness arise from the coefficients rather than the roots. Consider the Kac's model, 
%$p_n(z) = \sum_{k=0}^{n}a_kz^k,$ where the $a_k$ are i.i.d. complex Gaussian $N_{\mathbb {C}}(0,1).$ In this case, $p_n$ converges uniformly on compact subsets of the unit disk to the i.i.d. Gaussian power series $f(z)=\sum_{k\ge 0}a_kz^k$. Hence, we may expect the lemniscates $\Lambda_{n}$ to converge to the lemniscate $\Lambda_{f}=\{|f|<1\}$ in an appropriate sense. In particular, the inradius $\rho(\Lambda_{n})\to \rho(\Lambda_{f})$ in distribution. We omit details.
%

\section{Proof of Theorem \ref{Gin}}\label{sec:Ginibre}
 The idea of the proof proceeds along earlier lines: first we fix $t>0$ and show that $\log|p_n(z)|$ is negative w.o.p. for a fixed $z$ lying on $|z|= 1-t$.  It then follows from a net argument that the whole circle (and hence the disk) is contained in $\Lambda_n$ w.o.p. 

\vspace{0.1in}

\noindent Let $t\in (0,\frac{1}{100})$ and fix $z$ with $|z|= 1-t$. Taking logarithms, we have as before that
$$\log|p_n(z)| = \sum_{k=1}^{n}\log|z-X_k|,$$
except now the roots are no longer i.i.d. Define $F_t:\mathbb{C}\rightarrow\mathbb{R}$ by

\[F_t(w)=
\begin{cases}
       \log \frac{1}{t}, \hspace{0.05in} & |z-w|\geq  \frac{1}{t}\\
       \log|z-w|, \hspace{0.05in} & t < |z-w|< \frac{1}{t} \\
       \log t,   \hspace{0.05in}& |z-w|\leq t.
\end{cases}
\]
Next, we write
\begin{align*}\label{lip1}
\log|p_n(z)|&= \sum_{k=1}^{n}F_t(X_k) +\sum_{k: |z-X_k|\geq\frac{1}{t}}\left(\log|z- X_k| - \log\frac{1}{t}\right) +  \sum_{k: |z-X_k|\leq t}\left(\log|z- X_k| - \log t\right) \\
            &=: \mathcal{L}_1 + \mathcal{L}_2+ \mathcal{L}_3.
\end{align*}
Since the term $\mathcal{L}_3$ is negative, we have 
\begin{equation}\label{lip2}
\PP\left(\log|p_n(z)|\geq -\frac{t}{4} n\right) \leq \PP\left(\mathcal{L}_1 + \mathcal{L}_2 \geq - \frac{t}{4} n\right)
\end{equation}
We claim that the right hand side of \eqref{lip2} decays exponentially. For that we will need the following
\begin{prop}\label{prop2}
Fix $t>0$.
There exist constants $c_t, c_2 > 0$ such that for all large $n$, we have $$\PP\left(\mathcal{L}_1\geq -\frac{t}{2}n\right)\leq5\exp(-c_t n),\hspace{0.1in}\PP(\mathcal{L}_2\geq \frac{t}{4}n)\leq n\exp(-c_2n).$$
\end{prop}

\vspace{0.1in}

\noindent Assume the Proposition is true for now. Then, it is easy to see that the right hand side of \eqref{lip2} goes to $0$ exponentially with $n.$ Indeed,

\begin{align*}
\PP\left(\mathcal{L}_1 + \mathcal{L}_2 \geq - \frac{t}{4} n\right) & = \PP\left(\mathcal{L}_1 + \mathcal{L}_2 \geq - \frac{t}{4} n, \; \mathcal{L}_2< \frac{t}{4}n\right) + \PP\left(\mathcal{L}_1 + \mathcal{L}_2 \geq - \frac{t}{4} n, \; \mathcal{L}_2\geq \frac{t}{4}n\right)\\
&\leq \PP\left(\mathcal{L}_1 \geq - \frac{t}{2} n\right) + \PP\left(\mathcal{L}_2\geq \frac{t}{4}n\right)\\
&\leq 5\exp(-c_tn) + n\exp(-c_2n).
\end{align*}
which establishes the claim. We now proceed with the proof of Proposition \ref{prop2}. 
\begin{proof}[Proof of Proposition \ref{prop2}]

\vspace{0.05in}

\textbf{Step $1$: Estimate on $\mathcal{L}_2$} 

\vspace{0.05in}

\noindent Let $N_t = |\{k: |z- X_k|\geq\frac{1}{t} \}|$. If $\mathcal{L}_2\ge \frac{t}{4}n$, then we must have $N_t\ge 1$, which has probability at most $e^{-cn}$ for some $c>0$. To see this, let us recall the following fact about eigenvalues of the Ginibre ensemble.
% \commE{add book section to citation}.

%
%Since $\log\frac{1}{t} > 0$, we have 
%$$\PP(\mathcal{L}_2\geq \frac{t}{4}n)\leq\PP\left(\sum_{k: |z-X_k|\geq\frac{1}{t}}\log|z- X_k|\geq \frac{t}{4}n\right) $$ 
%Since $|z| < 1$ and $t$ is small, if $|z-X_k|\geq\frac{1}{t},$ we will have  $|z- X_k|\leq |2X_k|.$ Plugging this into the earlier estimate we obtain
%
%\begin{equation}\label{l2}
%\PP(\mathcal{L}_2\geq \frac{t}{4}n)\leq\PP\left(\sum_{k: |z-X_k|\geq\frac{1}{t}}\log |2X_k|\geq \frac{t}{4}n\right).
%\end{equation}
%
%\noindent Next, let $N_t = |\{k: |z- X_k|\geq\frac{1}{t} \}|.$ We claim that $\PP (N_t\geq 1)\leq n\exp(-c_3n).$ In order to show this, we will use the following lemma \cite{GAFbook} \commE{add book section to citation}
\begin{lemma}[Kostlan~\cite{kostlan}, \cite{GAFbook}] \label{lem:Lambda}
Let $\lambda_j$ be the eigenvalues (indexed in order of increasing modulus) of a Ginbre random matrix (un-normalized). Then, $$\{|\lambda_1|^2, |\lambda_2|^2,..., |\lambda_n|^2 \}\sim\{Y_1, Y_2,..., Y_n\},$$ 
where $Y_j$ is a sum of $j$ i.i.d. $Exp(1)$ random variables.
\end{lemma}

\noindent Now for the proof of the claim. Since $|z| < 1$ and $t\in (0, \frac{1}{100}),$ $|z-X_k|\geq\frac{1}{t}$ implies for instance that $|X_k| > 99.$ Therefore, by elementary steps and applying Lemma \ref{lem:Lambda}, we obtain
\begin{eqnarray}
\PP(N_t\geq 1)&\leq \PP(\max_{k}|X_k|\geq 99)\\
&=\PP(\max_{k}|X_k|^2\geq 99^2)\\
&=\PP(\max_{k}|\lambda_k|^2 > 99^2n)\\
&= \PP(\max_{k}Y_k > 99^2n),
\end{eqnarray}
where we have used $X_j = \dfrac{\lambda_j}{\sqrt{n}}$ in going from the second to third line above.
Then a union bound and a Cramer-Chernoff estimate gives
\begin{eqnarray}
\PP(\max_{k}Y_k > 99^2n) 
&\leq n\PP(Y_n > 99^2n) \\
&\leq n\exp(-c_2 n), \label{eq:chernoffbound}
\end{eqnarray}
and combining this with the above estimate we obtain
\be\label{eq:Nexp}
\PP(N_t\geq 1) \leq n\exp(-c_2 n),
\ee
as desired.
%Going back to \eqref{l2}, we can estimate 
%\begin{align*}
%\PP(\mathcal{L}_2\geq \frac{t}{4}n)&\leq\PP\left(\sum_{k: |z-X_k|\geq\frac{1}{t}}\log |2X_k|\geq \frac{t}{4}n\right) \\
%&=\PP\left(\sum_{k: |z-X_k|\geq\frac{1}{t}}\log |2X_k|\geq \frac{t}{4}n; N_t =1 \right) + \PP\left(\sum_{k: |z-X_k|\geq\frac{1}{t}}\log |2X_k|\geq \frac{t}{4}n; N_t > 1\right)\\
%&\leq\PP\left(\max_{k}\log|2X_k|\geq\frac{t}{4}n \right) + \PP(N_t > 1)\\
%&\leq C_4n\exp(-c_2n)
%\end{align*}
%In the final estimate above, we have used \eqref{eq:Nexp} to control the second term, and the first term decays exponentially by the same argument that was used to show \eqref{eq:Nexp}. We now move on to the estimate of $\mathcal{L}_1$

\vspace{0.1in}

\noindent \textbf{Step $2$: Estimate on $\mathcal{L}_1$}

\vspace{0.05in}

The desired estimate 
is equivalent to 
\begin{equation}\label{tail}
\PP\big\{\mathcal{L}_1-\mathbb{E}(\mathcal{L}_1)\geq -\frac{t}{2}n -\mathbb{E}(\mathcal{L}_1)\big\}\leq 5\exp(-c_tn).
\end{equation}
As preparation towards this, observe that
$\dfrac{1}{n}\mathbb{E}(\mathcal{L}_1) = \mathbb{E}\left( \int F_td\mu_n\right),$
where $\mu_n$ is the empirical spectral measure defined in \eqref{esd}. By the circular law of random matrices \cite{taovu}, almost surely $\mu_n$ and its expectation both converge to the uniform measure on the unit disk. As a result, taking into account that $F_t$ is a bounded continuous function, we obtain
\begin{equation}\label{circ}
\lim_{n\to\infty}\dfrac{1}{n}\mathbb{E}(\mathcal{L}_1) = \dfrac{1}{\pi}\int_{\mathbb{D}}F_t dm = \dfrac{|z|^2- 1}{2} + \dfrac{t^2}{2},
\end{equation}
 where the second equality in \eqref{circ} follows from a computation similar to the one in Example \ref{example}. Using $|z| = 1- t,$ the quantity on the right  reduces to $-t+t^2$. Hence, for large $n$, we have $\mathbb{E}(\mathcal{L}_1)\le -\frac34 tn$ and hence, if the event in 
\eqref{tail} holds, then
\begin{equation}\label{circ1}
\mathcal{L}_1-\mathbb{E}(\mathcal{L}_1)
\ge \frac{t}{4}n.
\end{equation}
Thus, our immediate goal is reduced to showing that the probability of the above event is at most $5\exp(-c_tn)$ for an appropriate constant $c_t$.
We invoke the following result of Pemantle and Peres~\cite[Thm. 3.2]{PP}.
\begin{thm}\label{PemPer}
Given a determinantal point process with $n < \infty$ points and $f$ a  Lipschitz-$1$ function on finite counting measures, for any $a>0$ we have
$$\mathbb{P}\left(|f- \mathbb{E}(f)| \geq a \right)\leq 5 \exp\left(-\frac{a^2}{16(a + 2 n)} \right).$$
\end{thm}
\noindent 
To say that $f$ is Lipschitz-$1$  on the space of finite counting measures means that
$$
\Big|f\left(\sum_{i=1}^{k+1}\delta_{x_i}\right)-f\left(\sum_{i=1}^k\delta_{x_i}\right)\Big|\le 1
$$
for any $k\ge 0$ and any points $x_1,\ldots ,x_k $.

In our case, as we have recalled, $\{X_1, X_2,...,X_n \}$ is a determinantal point process with exactly $n$ points. Moreover, $\mathcal{L}_1$ is  Lipschitz with Lipschitz constant $\|F_t\|_{\sup}=\log\frac{1}{t}$. Applying Theorem \ref{PemPer} to $\mathcal L_1/\log(1/t)$, we see that
\begin{align*}
\PP\big\{\mathcal{L}_1-\mathbb{E}(\mathcal{L}_1)\geq \frac{t}{4}n \big\}&= 5\exp\left\{-\frac{t^2n^2}{256(\log(1/t))^2(\frac{tn}{4\log(1/t)}+2n)} \right\}\\
&\le 5\exp\{-c_tn\}
\end{align*}
where we may take $c_t=c t^2/\log(1/t)^2$ for a large constant $c$. This completes the proof of the proposition.
%\PP\big\{\mathcal{L}_1-\mathbb{E}(\mathcal{L}_1)\geq n\left(-\frac{t}{2} -\frac{\mathbb{E}(\mathcal{L}_1)}{n}\right)\big\}\\
%&=\PP\big\{t\mathcal{L}_1-t\mathbb{E}(\mathcal{L}_1)\geq tn\left(-\frac{t}{2} -\frac{\mathbb{E}(\mathcal{L}_1)}{n}\right)\big\} \\
%&\leq 5\exp\left(-n^2t^2\frac{\left(-\frac{t}{2}- \frac{\mathbb{E}(\mathcal{L}_1)}{n}\right)^2}{16\left(tn\left(-\frac{t}{2} -\frac{\mathbb{E}(\mathcal{L}_1)}{n}\right)+ 2n \right)}\right).\\
%\end{align*}
%Note that by \eqref{circ1}, the quantity $-\frac{t}{2} -\frac{\mathbb{E}(\mathcal{L}_1)}{n}$ is positive for all large $n$, since $t$ is small.  If we factor out $n$ from the denominator in the above parentheses, and use \eqref{circ1} again, we obtain \eqref{tail}. This finishes the proof of the Proposition.
\end{proof}
Now that we have proved the pointwise estimate, the net argument from Lemma \ref{Bernstein} can be used to show that the whole circle $|z| = 1-t$ lies in the lemniscate w.o.p. The maximum principle then shows that the corresponding disk lies in the lemniscate w.o.p. This concludes the proof that $\Lambda_n$ contains $\D_r$ w.o.p.

We next prove that $\Lambda_n\subseteq \D_s$ w.o.p. for $s>1$. Fix $1<s'<s$ and let $\delta=s-s'$ and $\epsilon=\frac12 \log s$. We present the proof in four steps.
\begin{enumerate}
\item  $\frac{|\lambda_j|}{\sqrt{n}}<s'$ for all $j$, w.o.p., i.e., with probability at least $1-e^{-cn}$. To see this, invoke Lemma~\ref{lem:Lambda} to see that the complementary event has probability less than $ne^{-c(s')n}$ by the same reasoning used in \eqref{eq:chernoffbound}, noting that $99^2$ may be replaced by any constant greater than $1$.
\item Fix $z$ with $|z|=s$ and let $f_{z,\delta}(w)=\log\left(\min\{\max\{|z-w|,\delta\},\frac{1}{\delta} \}\right)$, a bounded continuous function. Then by \cite{hiaipetz} (Theorem $9$),
\begin{align*}
\PP\left\{\big| \frac{1}{n}\sum_{j=1}^nf_{z,\delta}(\lambda_j/\sqrt{n})-\int_{\D}f_{z,\delta}(w)\frac{dm(w)}{\pi} \big|>\epsilon\right\}\le e^{-c_{\epsilon,\delta}n^2}.
\end{align*}
\item On the event in (1), $f_{z,\delta}(\lambda_j/\sqrt{n})=\log|z-\frac{|\lambda_j}{\sqrt{n}}|$ for all $j$ and all $|z|=s$. Also, $f_{z,\delta}(w)=\log|z-w|$ for all $w\in \D$. Hence, w.o.p.
\begin{align*}
\PP\left\{\big| \frac{1}{n}\log|p_n(z)|-\log s \big|>\epsilon\right\}\le e^{-c_{\epsilon,\delta}n^2}+e^{-cn}.
\end{align*}
Hence, $\frac{1}{n}\log|p_n(z)|>\frac12 \epsilon$ w.o.p. by the choice of $\epsilon=\frac12\log s$.
\item Let $m=\frac{100}{\epsilon \delta}$ and let $z_1,\ldots ,z_m$ be equispaced points on $\partial \D_s$. Then w.o.p. $\inf_{j\le m} \frac{1}{n}\log|p_n(z_j)|>\frac12 \epsilon$ by the previous step. On the event in (1), $\|\nabla \frac{1}{n}\log|p_n(z)|\|\le \frac{1}{\delta}$, hence  $$\inf_{|z|=s} \frac{1}{n}\log|p_n(z)|>0$$ 
w.o.p. On this event $\Lambda_n\subseteq \D_s$.
\end{enumerate}
This concludes the proof of Theorem \ref{Gin}. 

\bibliographystyle{abbrv}
\bibliography{ref}

\end{document}